\documentclass[12pt]{amsart}
\usepackage{etex} \usepackage{amssymb,amsthm,amsmath,amsfonts,stmaryrd,soul}
\makeatletter{}\usepackage[OT2,OT1]{fontenc}
\newcommand\cyr
{
\renewcommand\rmdefault{wncyr}
\renewcommand\sfdefault{wncyss}
\renewcommand\encodingdefault{OT2}
\normalfont
\selectfont
}
\DeclareTextFontCommand{\textcyr}{\cyr}
\def\cprime{\char"7E }

\usepackage{mathabx}
\usepackage{bbm}
\newlength{\tmpl}

\usepackage{color}

\usepackage{pict2e}

\usepackage[margin=2cm,a4paper]{geometry}

\newcommand{\dd}[1]{d#1}

\usepackage{tikz,pgfplots,pgfplotstable}
\usepgfplotslibrary{statistics}

\newcommand\tikzsetnextfilename[1]{}

\usepackage[extra]{tipa}
\usepackage{graphicx}
\usepackage{caption,subcaption}

\usepackage{pst-node}
\usepackage{enumerate}

\numberwithin{equation}{section}

\newcommand{\onetwo}[1]{\hat{1}_2^{#1}} \newcommand{\onethree}[1]{\hat{1}_3^{#1}}
\newcommand{\pispecial}[1]{\nu_{0#1}}

\DeclareMathOperator{\Tr}{\mathrm{Tr}}

\newcommand{\dottie}{ \text{\itshape\textturnm}}
\let\Re\undefined
\let\Im\undefined
\DeclareMathOperator{\Re}{\mathrm{Re}}
\DeclareMathOperator{\Im}{\mathrm{Im}}
\DeclareMathOperator{\arccot}{arccot}
\DeclareMathOperator{\atanh}{\mathrm{atanh}}

\newcommand{\harpleftsign}{\scriptstyle\leftharpoonup}
\newcommand{\harpleft}[2]{  \ifx\displaystyle#1\doalign{$\harpleftsign$}{#1#2}\fi
  \ifx\textstyle#1\doalign{$\harpleftsign$}{#1#2}\fi
  \ifx\scriptstyle#1\doalign{\scalebox{.6}[.9]{$\harpleftsign$}}{#1#2}\fi
  \ifx\scriptscriptstyle#1\doalign{\scalebox{.5}[.8]{$\harpleftsign$}}{#1#2}\fi
}

 \newcommand{\harprightsign}{\scriptstyle\rightharpoonup}
\newcommand{\harpright}[2]{  \ifx\displaystyle#1\doalign{$\harprightsign$}{#1#2}\fi
  \ifx\textstyle#1\doalign{$\harprightsign$}{#1#2}\fi
  \ifx\scriptstyle#1\doalign{\scalebox{.6}[.9]{$\harprightsign$}}{#1#2}\fi
  \ifx\scriptscriptstyle#1\doalign{\scalebox{.5}[.8]{$\harprightsign$}}{#1#2}\fi
}

\newcommand{\doalign}[2]{ {\vbox{\offinterlineskip\ialign{\hfil##\hfil\cr#1\cr$#2$\cr}}}}

\def\R{{\mathbb R}}

\def\C{{\mathbb C}}
\def\IC{{\mathbb C}}
\def\IR{{\mathbb R}}
\def\IZ{{\mathbb Z}}
\def\IN{{\mathbb N}}
\def\N{{\mathbb N}}

\def\<{\langle}
\def\>{\rangle}

\def\E{{\mathbb E}}

\def\X{{\mathnormal X}}

\def\c{\underline c}

\def\s{\tilde{s}}
\def\Y{{\mathnormal Y}}

\def\A{\mathcal{A}}

\newcommand{\abs}[1]{
  \left\lvert
    #1
  \right\rvert}
\newcommand{\bigabs}[1]{
  \bigl\lvert
    #1
  \bigr\rvert}

\DeclareMathOperator{\Var}{\mathrm{Var}} \DeclareMathOperator{\SG}{\mathfrak{S}} \DeclareMathOperator{\NC}{\mathit{NC}} \DeclareMathOperator{\SP}{\mathcal{P}}   \DeclareMathOperator{\Rtrans}{\mathcal{C}}

\newtheorem{th-def}{Theorem-Definition}[section]
\newtheorem{theo}{Theorem}[section]
\newtheorem{lemm}[theo]{Lemma}

\newtheorem{prop}[theo]{Proposition}
\newtheorem{cor}[theo]{Corollary}
\theoremstyle{definition}

\newtheorem{notation}[theo]{Notation}

\newtheorem{Rem}[theo]{Remark}

\def\cvput#1[#2]{\pnode(#1,1){#1} \pscircle*(#1,1){.1} \rput(#1,.5){$#2$}}

\title{The Free  tangent law}\author[Wiktor Ejsmont]{Wiktor Ejsmont}

\address[Wiktor Ejsmont]{ 
Instytut Matematyczny, Uniwersytet Wroc\l awski,\\ pl.\ Grunwaldzki 2/4, 50-384 Wroc\l aw, Poland}
\email{wiktor.ejsmont@gmail.com}
\author[Franz Lehner]{Franz Lehner}
\address[Franz Lehner]{Institut für Diskrete Mathematik,
Technische Universität Graz,
Steyrergasse 30, 8010 Graz, Austria }
\email{lehner@math.tugraz.at}

\thanks{Supported by the Austrian Federal Ministry of Education, Science and
  Research and the Polish Ministry of Science and Higher Education, grants
  N$^{\textrm{os}}$ PL 08/2016 and PL 06/2018 and Wiktor
Ejsmont was supported by the Narodowe Centrum Nauki grant N$^{\textrm{o}}$ 2018/29/B/HS4/01420}


\begin{document}
\begin{abstract} 
Nevanlinna-Herglotz functions play a fundamental role for the study of
infinitely divisible distributions in free probability \cite{VoiculescuBercovici1992}. 
In the present paper we study the role of the tangent function,
which is a fundamental Herglotz-Nevanlinna function
\cite{Gesztesy,Donoghue:1974:monotone,Steinmetz},
and related functions in free probability.
To be specific, we show that the function
$$
\frac{\tan z}{1-x\tan z}
$$
of Carlitz and Scoville
\cite[(1.6)]{CarlitzScoville:1972} 
describes the limit distribution of
sums of free commutators and anticommutators and 
thus the free cumulants are given by the Euler zigzag numbers.

\end{abstract}

\date{
  \today}

\setlength{\parindent}{0pt}

\maketitle

\section{Introduction}
Nevanlinna or Herglotz functions are functions analytic in the upper half plane having non-negative imaginary part. This class has been thoroughly
studied during the last century and has proven very useful in many applications.
One of the fundamental examples  of Nevanlinna functions is the tangent function, see
\cite{Arlinskii,Gesztesy,Donoghue:1974:monotone,Steinmetz}. 
On the other hand it was shown by by Bercovici 
and Voiculescu \cite{VoiculescuBercovici1992}
that Nevanlinna functions characterize freely infinitely divisible
distributions. Such distributions naturally appear in free limit theorems
and 
in the present paper we show that the tangent function
appears in  a limit theorem for weighted  sums of free
commutators and  anticommutators. More precisely,
the family of functions
$$\frac{\tan z}{1-x\tan z}$$
arises,
which was studied by Carlitz and Scoville
\cite[(1.6)]{CarlitzScoville:1972} in connection with the
combinorics of tangent numbers;
in particular we recover  the   tangent function for $x=0$. 

In recent years a number of papers have investigated limit
theorems for the free convolution of probability measures  defined
by  Voiculescu \cite{Voiculescu:1986,Voiculescu:1991,VoiculescuDykemaNica:1992}.
The key concept of this definition is the notion of noncommutative
free independence, or freeness for short.
As in classical probability where the concept of independence gives rise to classical convolution, the concept of freeness leads to another operation on the
measures on the real line called free convolution. 
Many classical results in the theory of addition of independent random variables have their counterpart in this new
theory. 
For example the free analogue of the central limit theorem asserts 
that the distribution of
$$\frac{\X_1+\dots+\X_n}{\sqrt{n}},$$
for a given family of free identically distributed random variables converges 
in distribution to the normal law semicircle law as $n$ goes to infinity.
More general central limit theorems were proved by Speicher~\cite{Speicher:1992:central} by combinatorial means and provide the starting point
for the present paper.
We study limit theorems for sums with correlated entries,
more precisely, for
quadratic forms in free random variables.
In particular, we can explicitly compute the limit distribution
$\mu$ of the quadratic form
\begin{equation}
\label{eq:sumcommanticomm}
 \frac{1}{n} \sum_{k<l} a(X_kX_l+X_lX_k) + bi(X_kX_l-X_lX_k)
\end{equation}
where $a^2+b^2=1$ and $b\neq 0$,
and its   $R$-transform  turns out to be the elementary function
$$
R_{\mu}(z)=\frac{\tan(bz)}{b-a\tan(bz)}.
$$
This is the generating function
of the higher order tangent  numbers of
Carlitz and Scoville \cite{CarlitzScoville:1972}
which arise in connection with the enumeration of certain permutations. 
The central limit theorem for the mixed sum
of commutators and  anti-commutators  \eqref{eq:sumcommanticomm}  
will follow from a general  limit theorem for arbitrary quadratic forms.
The respective limit laws are infinitely divisible
and we call them the  \emph{free tangent law} and the \emph{free zigzag law}
according to the combinatorial interpretation of their cumulants.
In addition we indicate random matrix models for these limits.
The classical version of this limit theorem features
the $\chi^2$-distribution since  commutators trivially vanish
in classical probability.

\section{Preliminaries}
\label{sec:prelim}

\subsection{Basic Notation and Terminology}
A tracial noncommutative probability space is a pair $(\mathcal{A},\tau)$
where $\mathcal{A}$ is a von Neumann algebra, and  $\tau:\mathcal{A} \to
\mathrm{C}$ is a normal, faithful, tracial state, i.e., $\tau$ is linear and
continuous in the weak* topology, $\tau(\X \Y)=\tau(\Y \X)$,  $\tau(I)=1$,
$\tau(\X\X^*)\geq 0$ and $\tau(\X \X^{*}) = 0$ implies $\X = 0$ for all $\X,\Y
\in \mathcal{A}$.
For example, the noncommutative analog of a finite probability space
is the algebra of
complex $N\times N$ matrices $M_N(\IC)$. The unique tracial state is
the normalized trace $\tau_N(A)=\frac{1}{N}\Tr(A)=\frac{1}{N}\sum A_{ii}$.

The elements $\X\in{\mathcal{A}}$
are called (noncommutative) random variables; in the present paper
all random variables are assumed to be self-adjoint.
Given a noncommutative random variable $\X\in{\mathcal{A}}_{sa}$,
the spectral theorem provides a unique probability measure $\mu_X$ on 
$\R$ which encodes the distribution of $\X$ in the state $\tau$,
i.e., 
 $\tau( f(\X))=\int_\R f(\lambda)\,\dd{\mu_\X(\lambda)}$
for any bounded Borel function $f$ on $\R$.

\subsection{Free Independence}
A family of von Neumann subalgebras $\left(\mathcal{A}_i\right)_{i\in I}$ 
of $\mathcal{A}$ 
is called \emph{free}
 if $\tau(\X_{1} \dots \X_{n} ) = 0$ whenever $\tau(\X_{j} ) = 0$ for all
$j = 1,\dots, n$ and $\X_{j} \in \mathcal{A}_{i(j)}$ for some indices $i(1)\neq i(2)\neq \dots \neq i(n)$.
Random variables $\X_{1},\dots ,\X_{n} $  are freely independent (free) if the
subalgebras they generate are free. 
Free random variables can be constructed using the reduced free product
of von Neumann algebras \cite{Voiculescu:1985}.
For more details about free convolutions and free probability theory 
the reader can consult the standard references
\cite{VoiculescuDykemaNica:1992,NicaSpeicher:2006,MingoSpeicher:2017}.

\subsection{Free Convolution and the Cauchy-Stieltjes Transform}
It can be shown that the joint distribution of free random variables $X_i$
is uniquely determined by the distributions of the individual random variables
$X_i$ and therefore the operation of \emph{free convolution} is well defined:
Let $\mu$ and $\nu$ be probability measures on $\R$, and
$\X,\Y$ self-adjoint free random variables with respective distributions
$\mu$ and $\nu$.
The distribution of $\X+\Y$ is called the free additive convolution of $\mu$
and $\nu$ and is denoted by $\mu \boxplus \nu$.
The analytic approach to free convolution is based on the Cauchy transform
\begin{align}
\label{eq:hm:ball}
G_\mu(z)=\int_{\R}\frac{1}{z-y}\,\dd{\mu(y)}
\end{align}
of a probability measure $\mu$. The Cauchy transform is analytic on the upper half
plane $\IC^+=\{x+iy|x,y\in \R, y>0\}$ 
and takes values in the closed lower half plane
 $\IC^-\cup\IR$. For measures with compact support the Cauchy transform is analytic at infinity
and related to the moment generating function $M_{\X}$ as follows:
\begin{align}\label{mgf}
M_{\X}(z)=\sum_{n=0}^{\infty}\,\tau(\X^n)\,z^n = \frac{1}{z}\,G_\X (1/z).
\end{align}

Moreover the Cauchy transform has an inverse in some neighbourhood of infinity
which has the form
$$
G_\mu^{-1}(z) = \frac{1}{z} + R_\mu(z),
$$
where $R_\mu(z)$ is analytic in a neighbourhood of zero and is called 
\emph{$R$-transform}.
The coefficients of its series expansion
\begin{align} \label{rtr}
R_{\X}(z)=\sum_{n=0}^{\infty}\,K_{n+1}(\X)z^n
\end{align} 
are called \emph{free cumulants} of the random variable $\X$, see
Section~\ref{ssec:freecumulants} below.
The $R$-transform linearizes free convolution 
\begin{align} \label{freeconv}
  R_{\mu\boxplus\nu}(z)=R_\mu(z)+R_\nu(z)
  ,
\end{align}
see~\cite{Voiculescu:1986} and is the main tool for computation,
but for combinatorial purposes it will be convenient to consider the
shift $\Rtrans_{X}(z):=zR_{X}(z)$, which is called the
\emph{free cumulant transform} or \emph{free cumulant generating function}.

In order to treat measures with noncompact support, 
it is convenient to reformulate the identities in terms
of the reciprocal Cauchy transform $F_\mu(z)=1/G_\mu(z)$ 
\cite{BercoviciVoiculescu:1993:unbounded}.
This function has an analytic right compositional 
inverse $F_\mu^{-1}$ in a region
$$
\Gamma_{\eta,M}  = \{ z\in \IC \mid \abs{\Re z} < \eta\Im z, \text{ } \Im z > M\}
;
$$
the \emph{Voiculescu transform} is defined as the function
$$
\phi_\mu(z) = F_\mu^{-1}(z) - z
$$
which turns out to be $\phi_\mu(z) = R_\mu(1/z)$.

\subsection{Free infinite divisibility}
\label{ssec:FID}
In analogy with classical probability,
a probability measure $\mu$ on $\R$ is said to be
\emph{freely infinitely divisible} (or FID for short)
if for each $n \in \{1, 2, 3, \dots \}$ there exists a probability measure
$\mu_n$ such that
$\mu= \mu_n\boxplus\mu_n\boxplus\dots\boxplus\mu_n
$ ($n$-fold free convolution).

Free infinite divisibility of a measure $\mu$ is characterized
by the property that its Voiculescu transform has
a Nevanlinna-Pick representation
\cite{BercoviciVoiculescu:1993:unbounded}
\begin{equation}
  \phi_{\mu}(z)=
  \gamma+\int_{\R}\frac{1+xz}{z-x}\,\dd\rho({x})
  =
  \gamma+\int_{\R}\left(\frac{1}{z-x}+\frac{x}{1+x^2}\right)(1+x^2)\,\dd\rho({x})
\end{equation}
for some $\gamma \in \R$ and some nonnegative finite measure $\rho$.

We recall a general method to compute L\'evy measures from \cite{ArizmendiHasebe:2013}.
In terms of the free cumulant transform the L\'evy-Khintchine representation takes
the form \cite{BarndorffnielsenThorbjornsen:2006}
\begin{equation}\label{eq:LevyKhintchine}
\Rtrans_{\mu}(z)=cz+az^2+\int_{\R}\left(\frac{1}{1-xz}-1-xz\mathbf{1}_{\{\abs{x}<1\}}(x)\right)\dd{\nu(x)}
\end{equation}
for some $c \in \R$, $a \geq 0$ and a nonnegative measure $\nu$ satisfying
$\nu(\{0\}) = 0$ and $\int_{\R}\min\{1,x^2\} \dd{\nu(x)}<\infty.$ 
The triplet $(c,a,\nu)$ is called the \emph{free characteristic triplet}, $a$
is called the
\emph{semicircular
component} and $\nu$ is called the \emph{free L\'evy measure} of $\mu$.
The measure $\rho$ can be calculated using the Stieltjes inversion formula
\begin{equation*}
\int_{u}^v(1+x^2)\,\dd\rho({x})=-\frac{1}{\pi}\lim_{\epsilon \to 0^+}\int_{u}^v\Im\phi_\mu(x+i\epsilon)\,\dd{x}
\end{equation*}
for all points of continuity $u,v$ of $\rho.$ Considering the relation $R_{\mu}(z)=\phi(\frac{1}{z})$ and \eqref{eq:LevyKhintchine} we obtain $\frac{1+x^2}{x^2}\rho|_{\R\setminus \{0\}}=\nu|_{\R\setminus \{0\}}$ and $\rho(\{0\})=a.$
In particular, if the
function $-\frac{1}{\pi} \Im \phi_\mu(x+i\epsilon)$ converges uniformly to a continuous function $f_\mu(x)$ as $\epsilon \to 0^+$ on an interval $[u, v]$, then $\rho$ is absolutely continuous in $[u, v]$ with density $\frac{1}{1+x^2}f_\mu(x)$.
Hence, $\nu$ is also absolutely continuous in $[u, v]$ with density
$\frac{1+x^2}{x^2}f_\mu(x)$.
Regarding atoms, their mass is given by
\begin{equation}\label{eq:atoms}
\nu(\{x\}) =\frac{ 1}
{x^2} \lim_{\epsilon \to 0^+} i\epsilon \phi_\mu(x+i\epsilon), \quad x\in \R\setminus \{0\}.
\end{equation}

\subsection{Wigner semicircle law}
The Wigner semicircle law  has density
\begin{equation}
  \label{eq:semicirclelaw}
  \dd{\mu(x)} = \frac{1}{2\pi}
  \sqrt{4-x^2}\,\dd{x}
\end{equation}
on $-2 \leq x \leq 2 $. Its 
Cauchy-Stieltjes transform
is given by the formula
\begin{equation}
G_{\mu}(z)=\frac{z  -\sqrt{z ^2 - 4}}{2} , \label{eq:GtransformataMixner}
\end{equation}
where  $|z|$ is big enough and where the branch of the analytic square root is determined by the condition
that $\Im(z)>0\Rightarrow \Im(G_\mu(z))\leqslant 0$ (see \cite{SaitohYoshida:2001}). 

A non-commutative random variable $\X$ distributed according to the semicircle law
is called \emph{semicircular} or \emph{free gaussian} random variable.
The reason for the latter is the fact that its free cumulants $K_r=0$ for $r>2$
and it appears in the free version of the central limit theorem.

\subsection{Even elements}
 We call an element $\X\in \mathcal{A}$ \emph{even} if  all its odd moments vanish, i.e., $\tau(\X^{2i+1})=0$ for all $i\geq 0.$ 
It is immediate that the vanishing of all odd moments is
equivalent to the vanishing of all odd cumulants, i.e., $K_{2i+1}(\X)=0$  
and thus the even cumulants contain the complete information about the distribution of an even element. 

\subsection{Convergence in distribution }
\label{sssec:CLT}
In noncommutative probability
we say that a sequence $X_n$ of random variables \emph{converges in
distribution} towards $X$ as $n\to \infty$, denoted by
 $$\X_n
  \xrightarrow{d} X$$
  if we have for all $m\in\N$
   $$\lim_{n\to \infty}\tau(\X_n^m)=\tau( X^m)\text{ or equivalently}
   \lim_{n\to \infty}K_m(\X_n)=K_m(X)
   .$$

\subsection{Random matrices}
\label{ssec:randommatrix}
The semicircle law arises also as the asymptotic spectral distribution of
certain random matrices.
An $N\times M$ 
\textsl{complex Gaussian random matrix} is a  matrix  $X=[x_{i,j}]_{i,j=1}^{N\times M}$
whose entries form an i.i.d.\  complex Gaussian family with mean zero and
variance $\E(\abs{x_{i,j}}^2)=\frac{1}{N}$, i.e.,
the real parts $\Re x_{ij}$ and the imaginary parts $\Im x_{ij}$ together
form an i.i.d.\ family of $N(0,\frac{1}{2N})$ random variables.

An  $N\times N$ \textsl{GUE} random matrix is a matrix 
$Y_N=[y_{ij}]_{i,j=1}^{N\times N}$ of the form $Y_N=\frac{X_N+X_N^\ast}{\sqrt{2}}$
where $X_N$  is an  $N\times N$  complex Gaussian random matrix,
i.e., the family $\{y_{ii} \mid 1\leq i\leq N\}\cup \{\Re y_{ij} \mid 1\leq
i<j\leq N\}\cup\{\Im y_{ij} \mid 1\leq i<j\leq N\}$
is an independent family of real gaussian random variables with variance
$\Var y_{ii}=1/N$ and $\Var y_{ij}=\frac{1}{2N}$ for $i<j$.
It is well known that the moments  spectral distribution
a.s.\ converge to the moments of the standard Wigner semicircle law
\eqref{eq:semicirclelaw} 
$$
\lim_{N\to\infty} \tau_N(Y_N^m) = \frac{1}{2\pi}\int_{-2}^2 x^m\sqrt{4-x^2}\,\dd{x} =
\begin{cases}
  \frac{1}{n+1}\binom{2n}{n}&\text{if $m=2n$ is even,}\\
  0 &\text{otherwise,}
\end{cases}
$$
with respect to the normalized trace $\tau_N$.
In the language of section~\ref{sssec:CLT} this means that
$Y_N$ converges in distribution to a semicircular element with respect to the
expectation functional $\tau_N$.

\subsection{Convergence in eigenvalues}
Recently the concept of convergence with respect to the nonnormalized trace 
turned out to be useful 
for the study of the fine structure of random matrices
\cite{ColinsHasebeSakuma2018}. We say that a sequence of $N\times N$ 
deterministic 
matrices $A_N$ has limit distribution $\mu$ with respect to the
nonnormalized trace if the moments satisfy
$$
\lim_{N\to\infty}\Tr(A_N^m)=\int t^m \dd{\mu(t)}
$$
for every $m\in\IN$. The limit measure $\mu$ is not necessarily a probability measure,
but we explicitly assume that all moments of $\mu$ are finite
and 
thus the limit with respect to the normalized trace $\tau_N$ is zero.
Moreover the limit distribution is discrete
\cite[Proposition~2.10]{ColinsHasebeSakuma2018}
and under certain conditions the eigenvalues converge pointwise
\cite[Proposition~2.8]{ColinsHasebeSakuma2018}.

\subsection{Noncrossing Partitions}
We recall some facts about noncrossing partitions. For details and proofs see
the lecture notes \cite[Lecture~9]{NicaSpeicher:2006}.
Let $S\subseteq\N $ be a finite subset.
A partition of $S$ is a set of mutually disjoint subsets
(also called \emph{blocks}) $B_1,B_2,\dots,B_k\subseteq S$ 
whose union is $S$. Any partition $\pi$ defines an equivalence relation on $S$,
denoted by $\sim_\pi$, such that the equivalence classes are the blocks $\pi$. 
That is, $i\sim_\pi j$ if $i$ and $j$ belong to the same block of $\pi$.
A partition $\pi$ is called \emph{noncrossing} 
if different blocks do not interlace, i.e., there is no quadruple of
elements $i<j<k<l$ such that $i\sim_\pi k$ and $j\sim_\pi l$ 
but $i\not\sim_\pi j$. 

The set of non-crossing partitions of $S$ is denoted by $\NC(S)$,
in the case where $S=[n]:=\{1, \dots , n\}$ we write
$\NC(n):=\NC([n])$. 
$\NC(n)$ is a lattice under refinement order, where we declare $\pi\leq \rho$ if
every block of $\pi$ is contained in a block of $\rho$.
The subclass of noncrossing pair partitions (i.e., noncrossing complete matchings)
is denoted by $\NC_2(n)$.

The maximal element of $\NC(n)$ under this order is the partition consisting
of only one block and it is denoted by  $\hat{1}_{n}$.
On the other hand the minimal element $\hat{0}_n$ 
is the unique partition where every block is a singleton.
Sometimes it is convenient to visualize partitions as diagrams, for example
$\hat{1}_n=\makeatletter{}\begin{tikzpicture}[inner sep=0pt,scale=0.04]
\draw (2,0)--(2,7.5);
\draw (8,0)--(8,7.5);
\node at (18,3){$\cdots$};
\draw (26,0)--(26,7.5);
\draw (2,7.5)--(26,7.5);
\end{tikzpicture}
 $
and $\hat{0}_n=\makeatletter{}\begin{tikzpicture}[inner sep=0pt,scale=0.04]
\draw (2,0)--(2,4.5);
\draw (8,0)--(8,4.5);
\node at (18,2){$\cdots$};
\draw (26,0)--(26,4.5);
\draw (2,4.5)--(2,4.5);
\draw (8,4.5)--(8,4.5);
\draw (14,4.5)--(14,4.5);
\draw (20,4.5)--(20,4.5);
\draw (26,4.5)--(26,4.5);
\end{tikzpicture}
 $.

We will apply the product formula
\eqref{twr:produktargumentow} below only in the case
of pairwise products of random variables and in this case
two specific pair partitions and their complements will play a particularly important role,
namely the \emph{standard matching}
$\onetwo{n}=\makeatletter{}\begin{tikzpicture}[inner sep=0pt,scale=0.04]
\draw (2,0)--(2,4.5);
\draw (8,0)--(8,4.5);
\draw (14,0)--(14,4.5);
\draw (20,0)--(20,4.5);
\node at (30,2){$\cdots$};
\draw (38,0)--(38,4.5);
\draw (44,0)--(44,4.5);
\draw (2,4.5)--(8,4.5);
\draw (14,4.5)--(20,4.5);
\draw (38,4.5)--(44,4.5);
\end{tikzpicture}
 \in\NC(2n)$
and its shift $\pispecial{n}=\makeatletter{}\begin{tikzpicture}[inner sep=0pt,scale=0.04]
\draw (2,0)--(2,7.5);
\draw (8,0)--(8,4.5);
\draw (14,0)--(14,4.5);
\draw (20,0)--(20,4.5);
\draw (26,0)--(26,4.5);
\node at (36,2){$\cdots$};
\draw (44,0)--(44,4.5);
\draw (50,0)--(50,4.5);
\draw (56,0)--(56,7.5);
\draw (8,4.5)--(14,4.5);
\draw (20,4.5)--(26,4.5);
\draw (44,4.5)--(50,4.5);
\draw (2,7.5)--(56,7.5);
\end{tikzpicture}
 \in\NC(2n)$.

\subsection{Free Cumulants}
\label{ssec:freecumulants}
Given  a  noncommutative  probability space $(\A,\tau)$
 the free  cumulants are multilinear functionals $K_n : \mathcal{A}^n\to\IC$ 
defined  implicitly in terms of the mixed moments by the relation
\begin{align}
\tau(\X_{1}\X_{2}\dots \X_{n}) = \sum_{\pi \in \NC(n)}K_{\pi}(\X_{1},\X_{2},\dots ,\X_{n}),\label{eq:DefinicjaKumulant}
\end{align}
where 
\begin{align}
K_{\pi}(\X_{1},\X_{2},\dots ,\X_{n}):=\Pi_{B \in \pi}K_{\abs{B}}(\X_{i}:i \in B). \label{eq:DefinicjaProduktuKumulant}
\end{align}
 Sometimes we will abbreviate univariate cumulants as $K_{n}(\X)=K_{n}(\X,\dots ,\X )$.

Free cumulants provide a powerful technical tool to investigate
free random variables. 
This is due to the basic property of \emph{vanishing of mixed cumulants}. 
By this we mean the property that 
$$
K_n(X_1,X_2,\dots,X_n)=0
$$
for any family of random variables $X_1,X_2,\dots,X_n$ which can be partitioned
into two mutually free nontrivial subsets. 
For free sequences this can be reformulated as follows.
Let $(X_i)_{i\in\N}$ be a sequence of free random variables and
$h:[r]\to\N$ a map. We denote by $\ker h$ the set partition
which is induced by the equivalence relation 
$$
i\sim_{\ker h} j
\ 
\iff
\ 
h(i)=h(j)
.
$$
In this notation, vanishing of mixed cumulants implies that
\begin{equation}
  \label{eq:kerh>=pi}
  K_\pi(X_{h(1)},X_{h(2)},\dots,X_{h(r)})=0
  \text{ unless $\ker h\geq \pi$.}
\end{equation}

Our main technical tool is the free version,
due to Krawczyk and Speicher~\cite{KrawczykSpeicher:2000} (see also   \cite[Theorem 11.12]{NicaSpeicher:2006}), of the classical formula of 
James and Leonov/Shiryaev \cite{James:1958,LeonovShiryaev:1959} which
expresses cumulants of products in terms of individual cumulants.
\begin{theo}
\label{thm:krawczyk}
Let
$r,n \in \N$ and $ i_1 < i_2 < \dots < i_r = n$ be given and let
$$\rho=\{(1,2,\dots,i_1),(i_1+1,i_1+2,\dots,i_2),\dots,(i_{r-1}+1,i_{r-1}+2,\dots,i_r)\}\in \NC(n)$$ 
be the induced interval partition.
Consider now random variables $\X_1,\dots,\X_n\in\A$.
Then the free cumulants of the products can be expanded
as follows:
\begin{align} 
\label{twr:produktargumentow}
K_r(\X_1\dots \X_{i_1},\dots,\X_{i_{r-1}+1}\dots\X_n)=\sum_{\substack{\pi\in \NC(n) \\ \pi\vee\rho=\hat{1}_{n}} }K_\pi({\X_1,\dots,\X_n}). 
\end{align} 
\end{theo}

Our main tool is the following special case of  \cite[Proposition~4.5]{EjsmontLehner:2017}
which expresses cumulants of quadratic forms in 
free standard  semicircular variables and
takes the particularly simple form
\begin{equation}
  \label{eq:cumQnsemi}
  K_r(Q_n) = \Tr(A_n^r)
  ;
\end{equation}
thus the distributions of quadratic forms in semicircular variables
are easy to calculate, see \cite{EjsmontLehner:2019:commutators}.

\begin{notation}
For scalars $a,b, c\in\C$  we denote by  $\left[\begin{smallmatrix}
    c &a\\
    b& c
    \end{smallmatrix}\right]_n
\in M_n(\C)
$
the matrix whose diagonal elements are equal to $c$, whose  upper-triangular entries are equal to $a$ and
whose lower-triangular elements are equal to $b$, respectively. 
\end{notation}

\subsection{Combinatorics of tangent numbers}
\label{ssec:tangentnumbers}
The \emph{tangent numbers}
\begin{equation}
  \label{eq:tangentnumbers}
  T_{2k-1}=(-1)^{k+1}\frac{4^{k}(4^{k}-1){B_{2k}}}{2k}
\end{equation}
for $k\in\N $ are the Taylor coefficients of the tangent function
$$
\tan z = \sum_{n=1}^\infty T_n\frac{z^n}{n!} = z + \frac{2}{3!} z^3 +  \frac{16}{5!}z^5 +\frac{272}{7!}z^7 + \dotsm,
$$
see \cite[Page 287]{GrahamKnuthPatashnik:1994}.
The tangent numbers are complemented by the \emph{secant numbers}.  
Together they form
the sequence of $E_n$ of \emph{Euler zigzag numbers} which are the Taylor coefficients of the
function
$$
\tan z +\sec z=\sum_{n=0}^\infty
\frac{E_{n}}{n!}z^{n}.
$$ These numbers are also called \emph{up-down numbers} \cite{Carlitz:1975:permutations} or
\emph{snake numbers}  \cite{Arnold:1992:snakes,Hoffman:1999:derivative}
and  appear in several different contexts,
see for example \cite{Flajolet1980,Arnold:1991,StanleyVol2,Stanley2010} or  Andr\'e's theorem \cite{Andre1880}.

Similarly,
following Comtet~\cite[p.~260]{Comtet} (see also \cite{Cvijovic:2011:higher}) 
we define the \emph{arctangent numbers} by their exponential generating function
\begin{equation}
  \label{eq:def-arctannumbers}
  \frac{(\arctan z)^k}{k!} = \sum_{n=k}^\infty \frac{A_n^{(k)}}{n!}z^n;
\end{equation}
up to sign these are the same as the coefficients
of the hyperbolic arctangent function
\begin{equation}
  \label{eq:def-hyparctannumbers}
  \frac{(\atanh z)^k}{k!} = \sum_{n=k}^\infty \frac{\tilde{A}_n^{(k)}}{n!}z^n
\end{equation}
the latter are nonnegative and 
\begin{equation}
  \label{eq:signarctannumbers}
 A_n^{(k)} = (-i)^ki^n \tilde{A}_n^{(k)}.
\end{equation}

The \emph{higher order tangent numbers} $T_n^{(k)}$ were introduced by Carlitz and Scoville
\cite{CarlitzScoville:1972} as the coefficients of the Taylor series
$$\tan^{k+1} z =\sum_{n=k+1}^\infty T_n^{(k+1)}\frac{z^n}{n!}.$$
The generating function of the tangent polynomials $T_n(x) = \sum_{k=1}^nT_n^{(k)}x^k$ 
can be easily obtained from the geometric series
\begin{equation}
  \label{eq:genfunTnx}
\begin{aligned}
  T(x,z)
  &= \frac{x\tan z}{1-x\tan z}\\
  &= \sum_{k=1}^\infty x^k \tan^k z \\
  &= \sum_{n=0}^\infty \sum_{k=1}^n\frac{T_n^{(k)}}{n!} x^k z^n \\
  &= \sum_{n=0}^\infty \frac{T_n(x)}{n!} z^n 
\end{aligned}
\end{equation}
Note that our generating function  slightly differs from  Carlitz and Scoville's
\cite[Equation (1.6)]{CarlitzScoville:1972}, which is the expansion of the
function $\frac{\tan(z)}{1-x\tan(z)}$.

On the other hand it is well known that all derivatives of tangent and cotangent
can be expressed as certain polynomials, see  the side note \cite[Page 287]{GrahamKnuthPatashnik:1994})
and the recent studies \cite {Hoffman:1995,Hoffman:1999:derivative,Boyadzhiev:2007,Cvijovic:2009:derivative}.
To be specific, there is a sequence of polynomials $P_n(x)$ of degree
$n+1$, $n\geq 0$,
such that
\begin{align*}
  \frac{d^n}{d\theta^n}\tan\theta &= P_n(\tan\theta) \\
  \frac{d^n}{d\theta^n}\cot\theta &= (-1)^nP_n(\cot\theta)
\end{align*}
The generating function is easily derived from the Taylor series
$$
\tan(\theta+z) =  \sum_{n=0}^\infty \frac{P_n(\tan\theta)}{n!}z^n
= \frac{\tan\theta+\tan z}{1-\tan\theta\tan z}
$$
to be
$$
P(x,z) = \sum_{n=0}^\infty P_n(x)\frac{z^n}{n!} = \frac{x+\tan z}{1-x\tan z}
.
$$
Comparing the generating functions we find that
$$
xP(x,z) = (1+x^2)T(x,z) + x^2,
$$
and from this we conclude that
$$
xP_n(x) = (1+x^2)T_n(x),
$$
for $n\geq 1$, see also \cite {Cvijovic:2009:derivative}.
Note that $P_n(x)$ is divisible by $(1+x^2)$ because of the recurrence relation
$$
P_n(x)=(1+x^2)P_{n-1}'(x), \text{ } P_0(x)=x,
$$
see \cite[(6.95)]{GrahamKnuthPatashnik:1994}.

\subsection{An elementary lemma} 
The moments and the spectral measures  of the matrices of the underlying quadratic forms can be
computed explicitly and turn out to be connected to an old problem in classical
calculus. 
We first compute the eigenvalues of the matrix underlying
the quadratic form \eqref{eq:sumcommanticomm}.
\begin{lemm}\label{lemm:spectrum} Let $a,b\in\R$, $b\neq 0$ and 
$$ A_n=
\begin{bmatrix} 
  0  & a+bi & \dots  &  a+bi  \\ 
  a-bi &  0 & \dots & a+bi
  \\
  \multicolumn{4}{c}{\dotfill}\\
  a-ib &  a-bi & \dots  &   0
\end{bmatrix}\in M_n(\C).
$$
Then the eigenvalues of the matrix $A_n$ are given by 
$$
\lambda_k=b\cot\frac{\alpha+k\pi}{n}-a
,\text{ for } 0\leq k\leq n-1\text{ and }\alpha =\arccot(a/b). 
$$ 
\end{lemm}

\begin{proof}
The characteristic polynomial
$\chi_n(\lambda)=\det (\lambda I-A_n)$ 
satisfies the following recurrence relation. 
Let $w=a+bi=e^{i\alpha}$, the we have
\begin{align*}
  \chi_n(\lambda)
  &=
    \begin{vmatrix}
      \lambda &    -w &     -w & -w     &\dots& -w\\
      -\widebar{w}       &\lambda &     -w & -w     &\dots&-w\\      
      -\widebar{w}       & -\widebar{w}      &\lambda & -w     &\dots&-w\\      
      -\widebar{w}       & -\widebar{w}      & -\widebar{w}      &\lambda &\dots&-w\\      
      \multicolumn{6}{c}{\dotfill}\\
      -\widebar{w}       & -\widebar{w}      & -\widebar{w}      &  -\widebar{w}     &\dots&\lambda
    \end{vmatrix}
\intertext{we subtract the second row from the first row}
 &=
    \begin{vmatrix}
      \lambda +\widebar{w}&  -\lambda  -w &     0 & 0     &\dots& 0\\
      -\widebar{w}       &\lambda &     -w & -w     &\dots&-w\\      
      -\widebar{w}       & -\widebar{w}      &\lambda & -w     &\dots&-w\\      
      -\widebar{w}       & -\widebar{w}      & -\widebar{w}      &\lambda &\dots&-w\\      
      \multicolumn{6}{c}{\dotfill}\\
      -\widebar{w}       & -\widebar{w}      & -\widebar{w}      &  -\widebar{w}     &\dots&\lambda
    \end{vmatrix}
\intertext{and the   second column from the first column}
 &=
    \begin{vmatrix}
      2\lambda +w+\widebar{w}&  -\lambda  -w &     0 & 0     &\dots& 0\\
      -\lambda-\widebar{w}       &\lambda &     -w & -w     &\dots&-w\\      
      0       & -\widebar{w}      &\lambda & -w     &\dots&-w\\      
      0       & -\widebar{w}      & -\widebar{w}      &\lambda &\dots&-w\\      
      \multicolumn{6}{c}{\dotfill}\\
      0       & -\widebar{w}      & -\widebar{w}      &  -\widebar{w}     &\dots&\lambda
    \end{vmatrix}
  \\
  &= (2\lambda+w+\widebar{w})\chi_{n-1}(\lambda)-(\lambda+w)(\lambda+\widebar{w})\chi_{n-2}(\lambda)
\end{align*}
and the solution of this recurrence equation (with initial values
$\chi_0(\lambda)=1$
and $\chi_1(\lambda)=\lambda$) is
\begin{equation}
  \chi_n(\lambda) = \frac{w(\lambda+\widebar{w})^n-\widebar{w}(\lambda+w)^n}{w-\widebar{w}}
  .
\end{equation}
To compute the eigenvalues we may assume $\abs{w}=1$, i.e., $w=e^{i\alpha}$
and $\alpha =\arccot(a/b)$  (the general case follows by rescaling the matrix) and we substitute
$z=\lambda+w$. The matrix is selfadjoint and therefore any eigenvalue $\lambda$ is
real,
so $\widebar{z}=\lambda+\widebar{w}$ and we get
$$
w\widebar{z}^n-\widebar{w}z^n=0,
$$
i.e., $\Im(\widebar{w}z^n)=0$. Let $z=re^{i\theta}$, then this means
$$
\sin(n\theta-\alpha)=0
$$
and we conclude $\theta=\frac{\alpha+k\pi}{n}$.
We return to $\lambda=z-w=re^{i\theta}-e^{i\alpha}$. This is a real number and thus
the imaginary part vanishes, i.e., $r\sin\theta=\sin\alpha$, thus
$r=\frac{\sin\alpha}{\sin\theta}$
and finally
$$
\lambda = \sin\alpha\cot\theta - \cos\alpha 
$$
and in the general case where $w=a+ib$ the solutions are
\begin{equation}
  \label{eq:lambdak}
  \lambda_k = b\cot\frac{\alpha+k\pi}{n} - a,
  \quad 0\leq k\leq n-1.
\end{equation}
\end{proof}

\subsection{Cotangent sums}
The manipulations of the eigenvalues \eqref{eq:lambdak} will lead
to the following sums of cotangent powers which were
explicitly evaluated in our companion paper
\cite[Corollary~6.6]{EjsmontLehner:2019:cotangent}.
\begin{align}
  \label{eq:S2m+pi/2}
\sum_{k=1}^{n}\cot^{2m}\frac{(2k+1)\pi}{2n} 
    & = (-1)^mn 
    + \frac{1}{(2m-1)!}   \sum_{k=1}^m n^{2k} A_{2m}^{(2k)}\,T_{2k-1}\\
 \begin{split} \label{eq:Sm-pi/4}
\sum_{k=1}^n\cot^m\frac{(4k-1)\pi}{4n} 
    & = (-1)^{m/2}n \mathbbm{1}_{\text{$m$ even}}
    + \frac{1}{2(m-1)!}   \sum_{k=1}^m (-2n)^k A_{m}^{(k)} \, E_{k-1}\\
    &= \frac{(-2n)^m A_{m}^{(m)}E_{m-1}}{2(m-1)!}    
      + \mathcal{O}(n^{m-1}).
      \end{split}
\end{align}

\section{Limit theorems  and random matrix models for quadratic forms}

\subsection{A General  Limit Theorem}

In this section we consider  limit theorems for sums of commutators and other
quadratic forms of the following type.
\begin{theo}
  \label{thm:quadraticCLT}
  Let $A_n = [a_{i,j}^{(n)}]\in M_n(\IC)$ be a sequence of selfadjoint
  matrices such that
  $\sup_{i,j,n} \bigabs{a_{i,j}^{(n)}}<\infty$ and such that the matrix
  $\frac{1}{n}A_n$ has limit distribution $\mu$
  with respect to the nonnormalized trace.
  Let $X_i$ be free copies of a centered random variable $X$ of variance 1, then
  the sequence of quadratic forms
  $$
  Q_n  = \frac{1}{n}\sum_{i,j=1}^n a_{i,j}^{(n)} X_iX_j
  $$
  converges in distribution to $Y$, where   $$
  K_r(Y) =  \int
  t^r \dd{\mu(t)}.
  $$
\end{theo}
\begin{Rem}
  \label{rem:semicircle}
  From \cite[Proposition 2.10]{ColinsHasebeSakuma2018}, we conclude that the measure $\mu$ is discrete. 
  The limit measure in Theorem~\ref{thm:quadraticCLT}
  does not depend on the specific distribution of $X_i$ and therefore
  in the examples computed below we
  can replace the
  sequence $X_i$  by a free i.i.d.\  sequence of standard
  semicircular variables, which has the advantage that formula
  \eqref{eq:cumQnsemi} can be applied. 
\end{Rem}
\begin{proof}
  We use the product formula from Theorem~\ref{thm:krawczyk}:
  \begin{align*}
    K_r(Q_n)
    &=\frac{1}{n^{r}}
      \sum_{i_1,i_2,\dots,i_{2r}}  a_{i_1,i_2}^{(n)} a_{i_3,i_4}^{(n)} \dotsm a_{i_{2r-1},i_{2r}}^{(n)}
         K_r(X_{i_1}X_{i_2},X_{i_3}X_{i_4},\dots,X_{i_{2r-1}}X_{i_{2r}})
    \\
    &=\frac{1}{n^{r}}
      \sum_{i_1,i_2,\dots,i_{2r}}  a_{i_1,i_2}^{(n)} a_{i_3,i_4}^{(n)} \dotsm a_{i_{2r-1},i_{2r}}^{(n)}
      \sum_{\substack{\pi\in \NC(2r) \\ \pi\vee\makeatletter{}\begin{tikzpicture}[inner sep=0pt,scale=0.04]
\draw (2,0)--(2,4.5);
\draw (8,0)--(8,4.5);
\draw (14,0)--(14,4.5);
\draw (20,0)--(20,4.5);
\node at (30,2){$\cdots$};
\draw (38,0)--(38,4.5);
\draw (44,0)--(44,4.5);
\draw (2,4.5)--(8,4.5);
\draw (14,4.5)--(20,4.5);
\draw (38,4.5)--(44,4.5);
\end{tikzpicture}
 =\hat{1}_{2r}} }
         K_\pi(X_{i_1},X_{i_2},X_{i_3},X_{i_4},\dots,X_{i_{2r-1}},X_{i_{2r}})
    \\
    &=\frac{1}{n^{r}}
      \sum_{\substack{\pi\in \NC(2r) \\ \pi\vee\makeatletter{}\begin{tikzpicture}[inner sep=0pt,scale=0.04]
\draw (2,0)--(2,4.5);
\draw (8,0)--(8,4.5);
\draw (14,0)--(14,4.5);
\draw (20,0)--(20,4.5);
\node at (30,2){$\cdots$};
\draw (38,0)--(38,4.5);
\draw (44,0)--(44,4.5);
\draw (2,4.5)--(8,4.5);
\draw (14,4.5)--(20,4.5);
\draw (38,4.5)--(44,4.5);
\end{tikzpicture}
 =\hat{1}_{2r}} }
      \sum_{\ker\underline{i}\geq \pi}
         a_{i_1,i_2}^{(n)} a_{i_3,i_4}^{(n)} \dotsm a_{i_{2r-1},i_{2r}}^{(n)}
         K_\pi(X).
  \end{align*}
  By assumption $X$ is centered and therefore only partitions without 
  singletons contribute to this sum.
  Every block of such a partition $\pi$ has at least size 2
  and therefore $\abs{\pi}\leq r$. This in turn implies that there
  are only $n^{\abs{\pi}}$ allowed choices of indices $\underline{i}$
  and we have the following estimate
  \begin{align*}
    \abs{\frac{1}{n^{r}}
       \sum_{\ker\underline{i}\geq \pi}
         a_{i_1,i_2}^{(n)}a_{i_3,i_4}^{(n)} \dotsm a_{i_{2r-1},i_{2r}}^{(n)}
         K_\pi(X)}
    &\leq 
      n^{\abs{\pi}-r} C^r
      \abs{K_\pi(X)}
  \end{align*}
  where $C= \sup_{i,j,n} \abs{ a_{ij}^{(n)}}$.
  Now unless $\abs{\pi}=r$ this converges to zero as $n\to\infty$,
  on the other hand, $\abs{\pi}=r$ is only possible if $\pi$ is a pair partition.
  The only pair partition satisfying $\pi\vee\makeatletter{}\begin{tikzpicture}[inner sep=0pt,scale=0.04]
\draw (2,0)--(2,4.5);
\draw (8,0)--(8,4.5);
\draw (14,0)--(14,4.5);
\draw (20,0)--(20,4.5);
\node at (30,2){$\cdots$};
\draw (38,0)--(38,4.5);
\draw (44,0)--(44,4.5);
\draw (2,4.5)--(8,4.5);
\draw (14,4.5)--(20,4.5);
\draw (38,4.5)--(44,4.5);
\end{tikzpicture}
 =\hat{1}_{2r}$ is the partition $\pi=\pispecial{r}$  and finally we have
  $$
  K_r(Q_n) = \frac{1}{n^r}\Tr(A_n^r)K_2(X)^r + \mathcal{O}(1/n) \xrightarrow[n\to\infty]{} \int
  t^r \dd{\mu(t)}.
  $$
\end{proof}

\subsection{Random matrix models }
In this  subsection we construct random matrices whose limit law coincides with
the limit law from Theorem~\ref{thm:quadraticCLT}.
In some sense it is a simultaneous limit obtained from approximating the
semicircle law on the one hand
as in section~\ref{ssec:randommatrix}
and the free central limit law on the other hand.
To this end we consider compressions with random matrices.
In 
\cite[Proposition 12.18]{NicaSpeicher:2006} the authors
describe compound free Poisson distributions as free compressions with
semicircular operators.
The next proposition provides a complex version of this result, i.e.,
a description of compressions with circular operators.
Recall that a \textsl{circular operator}
is an operator $C$ of the form $C =
(X+iY)/\sqrt{2}$ 
where $X$ and $Y$ are free standard semicircular random variables.
\begin{prop} \label{prop:CykliczneVariancja-r}
  Let $C_1, C_2,\dots, C_n\in \A$ be a free family of circular random
  variables such that $K_2(C_i,C^\ast_i)=1$
  and pick  an arbitrary element
  $Z\in \A_{sa}$ which is free from the $C_i$.
  Let further $A=[a_{i,j}]_{i,j=1}^n\in M_n(\C)$ be a scalar selfadjoint matrix
  and put $T_n=\sum_{i,j}^na_{i,j}C_iZC_j^\ast$.
  Then the cumulants of $T_n$ are given by
\begin{align}  \label{eq:kumulantsamplevariancenotiid-r}
 K_r(T_n)= \Tr\otimes \tau([A
            \otimes Z]^r),
    \end{align}
where $A\otimes Z \in M_n(\C)\otimes \A$, with functional
$\Tr\otimes \tau$.
\end{prop}   
\begin{proof}
From the definition of $T_n$ and Theorem \ref{thm:krawczyk} we see that 
\begin{align*}
K_r(T_n)
& = \sum_{\substack{
     i_1,i_2,\dots,i_{2r} \in[n]  
      }}
    \sum_{\substack{\pi\in\NC(3r)\\ \pi\vee \onethree{r}=\hat{1}_{3r}}}
      a_{i_{1},i_{2}}   a_{i_{3},i_{4}} \dotsm       a_{i_{2r-1},i_{2r}}
      K_\pi(C_{i_{1}},Z,C^\ast_{i_{2}},C_{i_{3}},Z,C^\ast_{i_{4}},\dots,C_{i_{2r-1}},Z,C^\ast_{i_{2r}}).
 \intertext{Since $Z$ is free from the family $C_i$  every 
  partition with nonzero contribution can be written as $\pi=\rho\cup\sigma$
  where
  $\rho\in \NC_2(\{1,3,4,6,7,\dots,3r-2,3r\})$ is a pair partition and
  $\sigma\in\NC(\{2,5,\dots,3r-1\}$ is arbitrary.
  Now by the argument from the proof of   \cite[Proposition 12.18]{NicaSpeicher:2006} we conclude that the only
  pair partition satisfying the required condition is
  $\rho=\makeatletter{}\begin{tikzpicture}[inner sep=0pt,scale=0.04]
\draw (2,0)--(2,7.5);
\draw (8,0)--(8,4.5);
\draw (14,0)--(14,4.5);
\draw (20,0)--(20,4.5);
\draw (26,0)--(26,4.5);
\node at (36,2){$\cdots$};
\draw (44,0)--(44,4.5);
\draw (50,0)--(50,4.5);
\draw (56,0)--(56,7.5);
\draw (8,4.5)--(14,4.5);
\draw (20,4.5)--(26,4.5);
\draw (44,4.5)--(50,4.5);
\draw (2,7.5)--(56,7.5);
\end{tikzpicture}
 $, while $\sigma$ is arbitrary.
  The result is}
 &= \sum_{i_1,i_2,\dots,i_r\in[n]} a_{i_r,i_1} a_{i_1,i_2} \dotsm a_{i_{r-1,i_r}}
   \sum_{\substack{\sigma{}\in\NC(r)}}K_\sigma(Z)\\
&= \Tr(A^r)\tau(Z^r)
\end{align*} 
  which is the desired formula.
\end{proof}
Let us now introduce some random matrix models. For notation see
section~\ref{ssec:randommatrix}.
\begin{prop}\label{randommatrixmodel}
  Let $X_{N\times NM}$ be a complex Gaussian random matrix  of size $N\times NM$ and
  let $D_M = [d_{i,j}^{(M)}] $  be a sequence of selfadjoint deterministic $M\times M$ 
  matrices such that $D_N$  has limit distribution $\mu$
  with respect to the nonnormalized trace. 
  Then for any sequence $P_{N}$  of  $N\times N$ (selfadjoint)
  deterministic   matrices
  which converges to $Z$ 
  with limit distribution $\nu$
  we have
  $$
  X_{N\times NM}[D_M\otimes P_N]X_{N\times NM}^\ast
  \xrightarrow[M,N\to\infty]{d}
  Y,
  $$
  where
  \begin{align*}
  R_{Y}(z)&=\sum_{r=1}^\infty\int_\R x^{r}\dd{\mu(x)}\tau(Z^{r})z^r\\
    &= \int\int\frac{xt}{1-xtz}\,\dd{\mu(x)}\,\dd{\nu(t)}.
  \end{align*}
  
\end{prop}

  \begin{proof} 
 Fix $M$ and observe that we can represent random matrix as a quadratic form in $M$ variables by the formula 
 $$ X_{N\times NM}[D_M\otimes P_N] X_{N\times NM}^\ast=\sum_{i,j=1}^M d_{{i,j}}X_{i,N}P_NX_{j,N}^\ast,$$
where  $X_{i,N}$ are  complex Gaussian random matrices (non selfadjoint) of size $N\times N$.
From Voiculescu's asymptotic freeness results
  \cite{Voiculescu:1991} (see also \cite[Chapter 4]{MingoSpeicher:2017})
  we infer that $$\sum_{i,j=1}^M d_{{i,j}}X_{i,N}P_NX_{j,N}^\ast
  \xrightarrow[\substack{N\to\infty}]{d}\sum_{i,j=1}^M d_{{i,j}} C_iZC_{j}^\ast,$$ 
  where $C_i$ has circular distribution  and $C_i$ and  $Z$ are free.
  By  Proposition \ref{prop:CykliczneVariancja-r}, we have 
  \begin{align*}
  K_r(\sum_{i,j=1}^M d_{{i,j}}C_iZC_{j}^\ast)=\Tr(D_M^r)\tau(Z^r)
  \xrightarrow[\substack{M\to\infty}]{} \int_\R x^{r}\dd{\mu(x)}\tau(Z^r),
  \end{align*}
  which finishes the proof. 
  
\end{proof}

This implies the following random matrix model
for the limit law from Theorem~\ref{thm:quadraticCLT}.
It was used to produce the histograms in 
Figures~\ref{fig:tanlaw} and~\ref{fig:ziglaw}.
\begin{cor}\label{corrandommatrixmodel2}
  Let $X_{N\times NM}$ be  as  in Proposition~\ref{randommatrixmodel} and
 $A_M = [a_{i,j}^{(M)}] $ be a sequence of selfadjoint  $M\times M$ matrices  as  in Theorem~\ref{thm:quadraticCLT}.
 Let $P_{N}$  be a sequence of  $N\times N$  deterministic 
 matrices all of whose moments with respect to the normalized trace converge to
 $1$,
 e.g., the identity matrices  $P_{N}=\left[\begin{smallmatrix}
    1 &0\\
    0& 1
  \end{smallmatrix}\right]_N$ or any projection matrix of large  rank  like
$P_{N}=\left[\begin{smallmatrix}
    1 & 0\\
    0& 1
    \end{smallmatrix}\right]_N-\frac{1}{N}\left[\begin{smallmatrix}
    1 & 1\\
     1& 1
    \end{smallmatrix}\right]_N$,
  then the spectral measures of
  $$
  \frac{1}{M}X_{N\times NM}[A_M\otimes P_N]X_{N\times NM}^\ast 
  $$
  converge in distribution to the
  limit law described in Theorem~\ref{thm:quadraticCLT}.
        \end{cor}
Next we provide another random  matrix model with self-adjoint GUE matrices. 
\begin{prop}\label{prop:randommatrixmodelnonnormalized}
  Let $X_N$ be standard random matrix from the GUE of size $N\times N$ and
  let $D_N  \in M_N(\IC)$ be a sequence of selfadjoint deterministic
  matrices such that $D_N$  has limit distribution $\mu$
  with respect to the nonnormalized trace.
  Then the random matrix sequence $ X_ND_NX_N$  converges to the measure $\mu$ with respect to the
  nonnormalized trace.
\end{prop}
\begin{Rem}
  First observe that the preceding result is a special case of \cite[Theorem 5.1 (i), $k=1$]{ColinsHasebeSakuma2018}, but our proof is different.
  On the other hand,
  the spectral measures of $X_ND_NX_N$ converge to zero with respect to the normalized trace.
  Indeed $\lim_{N\to \infty}\Tr(D_N^m)/N=0$ and a sequence of standard GUE matrices is almost surely uniformly
  bounded. The point here is that with respect to the nonnormalized trace
  $\Tr(\cdot)$  we obtain interesting limits.
\end{Rem}

In order to prove  Proposition \ref{prop:randommatrixmodelnonnormalized} we will
refer to a combinatorial result from random matrix theory, which we rewrite in
terms of the nonnormalized trace. To formulate this result we need the
following notation.
\begin{notation}
  \label{def:pairpartitions}
  \begin{enumerate}[1.]
   \item We denote by $\SP_2(m)$ the set of pair partitions,
    i.e., partitions of $\{1,2,\dots,m\}$ into blocks of size $2$;
    this set is empty unless $m$ is even.
   \item Let $\pi\in\SP_2(m)$ be a pair partition.
    To each block  $\{i,j\}\in\pi$ we associate the transposition $(i\,j)$
    and we identify the pair partition $\pi$ with the permutation obtained
    as the product of these transpositions. Since they are disjoint, this
    permutation is well defined.
   \item Let $\sigma\in\SG_n$ be a permutation and
    $\sigma=\gamma_1\gamma_2\dots\gamma_r$ be its cycle decomposition.
    Then for any family of matrices $A=(A_1,A_2,\dots,A_n)$ we denote by
    $$
    \Tr_\sigma(A_1,A_2,\dots,A_n) = \Tr_{\gamma_1}(A)\Tr_{\gamma_2}(A)\dotsm\Tr_{\gamma_r}(A)
    $$
    where for a cycle $\gamma=(i_1\,i_2\,\dotsm\, i_k)$ the cyclic trace is
    $$
    \Tr_\gamma(A) = \Tr(A_{i_1}A_{i_2}\dotsm A_{i_k}).
    $$
  \end{enumerate}

\end{notation}

\begin{prop}{\cite[Proposition~22.32]{NicaSpeicher:2006}}
  \label{prop:ranodmmatrixpropSpeicher}
  Let $X_N$ be a standard $N\times N$ GUE matrix as in
  Proposition~\ref{prop:randommatrixmodelnonnormalized}
  and $D$  be a constant $N\times N$ matrix.

  Then we have for all $m\in \N$, and  all $q_1,\dots, q_m \in \N$, that 
  $$
  \Tr\otimes \E(X_N D^{q_1}\dots X_N D^{q_m})=\sum_{\pi\in
    \SP_2(m)}\Tr_{\pi\gamma}(D^{q_1},\dots,D^{q_m})N^{-m/2},
  $$
  where
  $\gamma\in\SG_m$ is the cyclic permutation with one cycle
  $\gamma=(1,2,\dots,m)$, $\pi\gamma$ is the composition of this cycle with the permutation $\pi$
  associated to the pair partition according to Notation~\ref{def:pairpartitions}.
\end{prop}
\begin{proof}[Proof of Proposition 3.6]
 The $m$-th nonnormalized moment  of $ X_ND_NX_N$ is then given by
 \begin{align*}
 \Tr\otimes \E[( X_ND_NX_N )^m]&=\Tr\otimes \E(\underbrace{X_N D_N  X_N I_N\dots X_N D_N X_N I_N}_{m-times})
 \intertext{where $I_N$ is the identity matrix of size $N\times N$. Put $D=D_N$ in Proposition~\ref{prop:ranodmmatrixpropSpeicher}, then $D^0=I_N$. The advantage of this
interpretation becomes apparent from the fact that in this language we
can rewrite our last equation as }
& =\sum_{\pi\in \SP_2(2m)}\Tr_{\pi\gamma}(D_N ,I_N,\dots,D_N ,I_N)N^{-m}.
 \end{align*}Now let us look at the asymptotic structure of this formula.
 We have to determine the cycles of the permutation 
$\pi\gamma$ which asymptotically  contribute a non-zero factor. Recall that
by assumption $\lim_{N\to \infty}\Tr(D_N ^m)$ exists for all $m \in \N$. In
this situation the factor $N^{-m}$ is cancelled  if and only if $\pi\gamma$ contains exactly
the $m$  singleton cycles $(2),\dots,(2m)$ and each of them contributes the factor $\Tr(I_N)=N$.
This happens if and only if $\pi=\makeatletter{}\begin{tikzpicture}[inner sep=0pt,scale=0.04]
\draw (2,0)--(2,7.5);
\draw (8,0)--(8,4.5);
\draw (14,0)--(14,4.5);
\draw (20,0)--(20,4.5);
\draw (26,0)--(26,4.5);
\node at (36,2){$\cdots$};
\draw (44,0)--(44,4.5);
\draw (50,0)--(50,4.5);
\draw (56,0)--(56,7.5);
\draw (8,4.5)--(14,4.5);
\draw (20,4.5)--(26,4.5);
\draw (44,4.5)--(50,4.5);
\draw (2,7.5)--(56,7.5);
\end{tikzpicture}
 $.  Indeed in
order to generate the singleton cycle $(2)$, the partition $\pi$ must contain
the pair $\{2,3\}$. To generate the cycle $(4)$, the pair $\{4,5\}$ must occur
in $\pi$ and so on.
It follows that  asymptotically the only non-zero contribution comes from the
pair partition $\pi=\makeatletter{}\begin{tikzpicture}[inner sep=0pt,scale=0.04]
\draw (2,0)--(2,7.5);
\draw (8,0)--(8,4.5);
\draw (14,0)--(14,4.5);
\draw (20,0)--(20,4.5);
\draw (26,0)--(26,4.5);
\node at (36,2){$\cdots$};
\draw (44,0)--(44,4.5);
\draw (50,0)--(50,4.5);
\draw (56,0)--(56,7.5);
\draw (8,4.5)--(14,4.5);
\draw (20,4.5)--(26,4.5);
\draw (44,4.5)--(50,4.5);
\draw (2,7.5)--(56,7.5);
\end{tikzpicture}
 $ which produces the
permutation $\pi\gamma=(1,3,\dots,2m-1)(2)(4)\dots (2m)$, and thus
 \begin{align*}\lim_{N\to \infty}\Tr\otimes \E(( X_ND_NX_N )^m)&=\lim_{N\to \infty}\Tr_{(1,3,\dots,2m-1)(2)(4)\dots (2m)}(D_N ,I_N,\dots,D_N ,I_N)N^{-m}\\&=\lim_{N\to \infty}\Tr(D_N^m)\times N^{m}\times N^{-m}=\int_{\R}x^m\dd{\mu(x)}. 
  \end{align*}
\end{proof}
\begin{cor}
  Let 
 $A_N = [a_{i,j}^{(N)}] \in M_N(\IC)$ be  as  in Theorem~\ref{thm:quadraticCLT}, then
   the spectral measures of  $
     \frac{1}{N}X_{N}A_NX_{N}
    $   converge with respect to the nonnormalized trace to the measure $\mu$.
        \end{cor}

\section{Limit Theorem  of Sums of Commutators and Anticommutators }
We will now illustrate the  limit theorem~\ref{thm:quadraticCLT}
with some interesting computable cases and start with sums of commutators and
anticommutators, the most general expression being \eqref{eq:sumcommanticomm}.

\subsection{A  Limit Theorem for commutators and anticommutators}
The main contribution of this paper is the following limit theorem
featuring the fundamental generating function of Carlitz and Scoville
\cite[(1.6)]{CarlitzScoville:1972}. 
\begin{theo}[Free generalized tangent law]\label{twr:CLTCommutatorsAniticommutators}
  Let $\X_1, \X_2,\dots, \X_n \in \A_{sa}$ be free centered copies of
  a random variable with variance $1$,
  then for any $a$, $b\in\IR$  with $a^2+b^2=1$ and $b\neq 0$,
  the limit law
  $$Q_n=\frac{1}{n}\sum_{\substack{k,j=1\\ k<j}}^n\big(a(\X_k\X_j+\X_j\X_k)+ib(\X_k\X_j-\X_j\X_k)\big)
  \xrightarrow{d} Y,$$
  has $R$-transform
  $$R_{Y}(z)=\frac{\tan(bz)}{b-a\tan(bz)}
    .
  $$
  The free cumulants are given by
   $$
  K_r(Y)
  = b^{r-1}\frac{T_r(a/b)}{r!}
  =\frac{b^{r}a}{r!} P_r(a/b)
  =
  (-1)^r\frac{b^{r}a}{r!}\cot^{(r)}(\alpha)
  .
  $$
                 where $\alpha=\arccot(a/b)$ and the polynomials
  $P_r(x)$, ${T_r(x)}$ were defined in Section \ref{ssec:tangentnumbers}. 
  
      \end{theo}

\begin{proof}
  The system matrix is $\frac{1}{n}A_n=
  \frac{1}{n}\left[\begin{smallmatrix}
    0 &a+ib\\
    a-ib& 0
    \end{smallmatrix}\right]_n$ from Lemma~\ref{lemm:spectrum}
  and its characteristic polynomial   is 
  $$
  \chi_n(\lambda)=
  \frac{w(\lambda+\frac{\widebar{w}}{n})^n - \widebar{w}(\lambda+\frac{w}{n})^n  }{w-\widebar{w}}
  $$
  where $w=a+bi$.
  The cumulant generating function 
  \begin{equation*}
    R_{Q_n}(z) =     \sum_{k=1}^\infty \frac{\Tr(A_n^k)}{n^k} z^{k-1},
  \end{equation*}
  can be obtained from the logarithmic derivative of the characteristic
  polynomial.
  Indeed if we factorize the characteristic polynomial
  $\chi_n(\lambda) = \prod_{i=1}^n (\lambda - \lambda_i)$ then
  \begin{equation*}
    \frac{\chi_n'(\lambda)}{\chi_n(\lambda)}
    =\sum_{i=1}^n  \frac{1}{\lambda-\lambda_i}
  \end{equation*}
  and
  $$
  \frac{1}{z}     \frac{\chi_n'(1/z)}{\chi_n(1/z)} = 
  \sum_{k=0}^\infty \sum_{i=1}^n \lambda_i^k z^k
  = n + zR_{Q_n}(z)
  .
  $$
  In our case
  $$
  \frac{\chi_n'(\lambda)}{\chi_n(\lambda)} = 
  n
  \frac{w(\lambda+\frac{\widebar{w}}{n})^{n-1} - \widebar{w}(\lambda+\frac{w}{n})^{n-1}  }  {w(\lambda+\frac{\widebar{w}}{n})^n - \widebar{w}(\lambda+\frac{w}{n})^n  }
  $$
  and
  \begin{align*}
    R_{Q_n}(z) 
    &= \frac{1}{z}\left(\frac{1}{z}\frac{\chi'_n(1/z)}{\chi_n(1/z)} -n\right)\\
    &=
    \frac{n}{z}
    \left(
    \frac{w(1+\frac{z\widebar{w}}{n})^{n-1} - \widebar{w}(1+\frac{zw}{n})^{n-1}  }    {w(1+\frac{z\widebar{w}}{n})^n - \widebar{w}(1+\frac{zw}{n})^n  }
    -1
    \right)
    \\
    &=
    -\abs{w}^2
    \frac{(1+\frac{z\widebar{w}}{n})^{n-1} - (1+\frac{zw}{n})^{n-1}  }    {w(1+\frac{z\widebar{w}}{n})^n - \widebar{w}(1+\frac{zw}{n})^n  },
  \end{align*}
  and the limit is
  \begin{align*}
    \lim_{n\to\infty} R_{Q_n}(z) &=R_Y(z)=-\abs{w}^2\frac{e^{z\widebar{w}}-e^{zw}}{we^{z\widebar{w}}-\widebar{w}e^{zw}},
 \intertext{and finally substituting $w=a+ib$ ($\abs{w}=1$), we get } &=\frac{-\exp(z(a-ib))+\exp(z(a+ib))}{(a+ib)\exp(z(a-ib))-(a-ib)\exp(z(a+ib))}
\\&=\frac{2i\sin(bz)}{-2i(a\sin(bz)-b\cos(bz))}=\frac{\tan(bz)}{b-a\tan(bz)}.
  \end{align*}
  Thus the $R$-transform can be expressed in terms of the generating function of the higher order tangent numbers
  \eqref{eq:genfunTnx} as $R(z) = \frac{1}{a}T(a/b,bz)$. The rest follows from simple manipulations
  using the combinatorics of tangent numbers discussed in Section~\ref{ssec:tangentnumbers}.
\end{proof}

\begin{Rem}
  There is another proof in terms of Newton's identities,
  also known as the Newton-Girard formulae,
  which provide a relation between two
  types of symmetric polynomials, namely between power sums and elementary
  symmetric polynomials.
  Observe that
  \begin{equation*}
\chi_n(\lambda)=\sum_{j=0}^n\lambda^j
                 \binom{n}{j}
\frac{(a+ib)(a-ib)^j-(a-ib)(a+ib)^j}{2ib}
                 =: \sum_{j=0}^n\lambda^j c_j,
  \end{equation*}
 whose $n$ zeros are the numbers $\lambda_k=b\cot\frac{\alpha+k\pi}{n}-a$, $k\in\{0,\dots,n-1\}$. Let $s_r^n=\lambda_1^r+\dots +\lambda_n^r.$ By Newton's formulas for roots of a polynomial, we have for $r\in\N$

 \begin{equation*}
s_{r}^n+s_{r-1}^nc_1+\dots+s_1^nc_{k-1}+kc_k=0
 \end{equation*}
 for $k\in\{1,\dots,n+1\}$.
Dividing both sides of above equation by $n^r$ 
and passing with $n$ to infinity for every fixed $k$ we get 
\begin{equation*}
\sum_{j=0}^{k-1}\s_{r-j}
  \frac{(a+ib)(a-ib)^j-(a-ib)(a+ib)^j}{2ibj!}
+k
\frac{(a+ib)(a-ib)^k-(a-ib)(a+ib)^k}{2ibk!}
=0,
\end{equation*}
where $\s_r=\lim_{n\to \infty}s_r^n/n^r$. Recall that ${R}_Y(z)=\sum_{r=0}^\infty\s_{r+1}z^r$ and by using Cauchy product of two infinite series, we see
\begin{multline*}
R_Y(z)\left(1+\frac{(a+ib)}{2ib}(\exp(z(a-ib))-1)-\frac{(a-ib)}{2ib}(\exp(z(a+ib))-1)\right)\\=-\frac{(a+ib)(a-ib)}{2ib}\exp(z(a-ib))+\frac{(a-ib)(a+ib)}{2ib}\exp(z(a+ib)),
\end{multline*}
which after a simple computation can be written in the  desired form. 
 \end{Rem}

 \begin{Rem}
    From Proposition~\ref{prop:randommatrixmodelnonnormalized} and Theorem~\ref{twr:CLTCommutatorsAniticommutators} 
 for $C_N=\frac{1}{N}X_{N}\left[\begin{smallmatrix}
    0&a+bi\\
    a-ib& 0
    \end{smallmatrix}\right]_N X_{N}$ we obtain a random matrix approximation  of the following moment generating function (with respect to the non-normalized trace)
                 $$\lim_{N\to \infty }M_{C_N}(z)=1+\frac{z\tan(bz)}{b-a\tan(bz)}.$$
 \end{Rem}
 \subsection{The free  tangent and zigzag laws }
In this subsection we indicate  yet another method to prove the limit
theorem~\ref{twr:CLTCommutatorsAniticommutators}  in some special cases,
namely sums of commutators and anticommutators.
These are interesting because up to rescaling the limit cumulants are
equal to the tangent numbers and   Euler's zigzag numbers. 
Secondly, these cumulants are reminiscent of certain formulae
for positive integer moments random variables related to the zeta function  \cite[see last line of Table 1]{BianePitmanYor2001}.
Thirdly, we incidentally solve a problem stated in \cite{CarlitzScoville:1972}.
According to  Theorem~\ref{thm:quadraticCLT}, in the above proofs
we  can restrict our sums to pair partitions and then the sums are simply
traces of powers of the matrix.
\begin{prop}[Free tangent law]
  \label{cor:tangentlaw}
 Let $\X_1, \X_2,\dots, \X_n \in \A_{sa}$ be free copies of
a random variable with  variance $1$,
 then 
 $$
 Q_n=\frac{1}{n}\sum_{\substack{k,j=1\\ k<j}}^n i(\X_k\X_j-\X_j\X_k)
 \xrightarrow{d} Y,
 $$
where $R_{Y}(z)=\tan(z)$. 
 We call the limit law  $\mu_Y$ the \emph{free tangent law}.
\end{prop}
\begin{proof}
  First observe that by virtue of
  the cancellation phenomenon see \cite[Theorem 4.4]{EjsmontLehner:2019:commutators}, we may assume without loss of generality that
  $X_i$ are even random variables and moreover by Remark~\ref{rem:semicircle}
  that they are semicircular.
  Thus the cumulants  can be computed using formula \eqref{eq:cumQnsemi} 
and evaluate to
$$K_r\big( \sum_{\substack{k,j=1\\ k<j}}^n i(\X_k\X_j-\X_j\X_k)\big)=\Tr(A_n^r)
\text{ where } 
A_n=\left[\begin{smallmatrix}
    0 &i \\
    -i& 0
    \end{smallmatrix}\right]_n. 
$$
The eigenvalues of the matrix $A_n$ were computed in Lemma~\ref{lemm:spectrum}  and they  are
$
\lambda_k=\cot\left(\frac{\pi}{2n}+\frac{k}{n}\pi\right)$ for $k\in\{0,\dots,n-1\} 
$ (including repeated eigenvalues), 
hence the odd cumulants vanish and the even
cumulants evaluate to
\begin{align*}
  K_{2m}\big(\sum_{\substack{k,j=1\\ k<j}}^n i(\X_k\X_j-\X_j\X_k)\big)
  &=\sum_{k=0}^{n-1} \cot^{2m}\left(\frac{\pi}{2n}+\frac{k}{n}\pi\right)\\
      & = (-1)^mn 
    + \frac{1}{(2m-1)!}   \sum_{k=1}^{m} n^{2k} A_{2m}^{(2k)}\,T_{2k-1}\\
    &= n^{2m} \frac{T_{2m-1}}{(2m-1)!} + \mathcal{O}(n^{2m-2})
\end{align*}
where we used formula   \eqref{eq:S2m+pi/2}, with $A_{2m}^{(2m)}=1$.
Hence  
\begin{equation*}
\lim_{n\to \infty}K_{2m}(Q_n)= \frac{T_{2m-1}}{(2m-1)!}
\end{equation*}
and we conclude that
 $$\lim_{n\to \infty}R_{Q_n}(z)=
 \tan(z).$$
\end{proof}

\begin{prop}[Free zigzag law]
 Let $\X_1, \X_2,\dots, \X_n \in \A_{sa}$ be free  copies 
of a centered  random variable with  variance $1$,
 then 
 $$Q_n=\frac{1}{{2}n}
 \sum_{\substack{k,l=1\\ k<l}}^n\big(\X_k\X_l+\X_l\X_k+i(\X_k\X_l-\X_l\X_k)\big)
  \xrightarrow{d} Y,$$
  where $R_{Y}(z)=\frac{1}{2}(\tan(z)+\sec(z)-1)$.
  The density of this law is shown in Fig.~\ref{fig:ziglaw}.
\label{twr:zigzag} 
\end{prop}
\begin{proof}

The matrix $A_n=\frac{1}{2}\left[\begin{smallmatrix}
    0 &1-i \\
    1+i& 0
    \end{smallmatrix}\right]_n$ 
corresponds to  $\alpha=\arccot(-1)=-\frac{\pi}{4}$ and
Lemma~\ref{lemm:spectrum} yields
$$
\lambda_k=-\frac{1}{{2}}\cot\left(-\frac{\pi}{4n}+\frac{k}{n}\pi\right)-\frac{1}{{2}}, \text{ for }
  k\in\{1,\dots,n\},
$$ 
where range of the index variable $k$ is shifted to $\{1,\dots,n\}$.
By the binomial  theorem applied  for  $r\geq 2$, we see that
\begin{align*}
\sum_{k=1}^n\left(-\cot\left(-\frac{\pi}{4n}+\frac{k}{n}\pi\right)-1\right)^r&=(-1)^{r}\sum_{j=0}^r {r\choose j}\left(\sum_{k=1}^n\cot^{r-j}\left(-\frac{\pi}{4n}+\frac{k}{n}\pi\right)\right)
\\
&=\frac{E_{r-1}}{(r-1)!}2^{r-1}n^{r}+\mathcal{O}(n^{r-1})
.\end{align*}
by  \eqref{eq:Sm-pi/4}.

Finally for  $r\geq 2$, we get 
\begin{align*}
\tilde{K}_r&=\lim_{n\to
  \infty}K_r(Q_n)
  \\
  &=\lim_{n\to\infty}\frac{1}{n^r}\sum_{k=1}^n\lambda_k^r=\lim_{n\to\infty}\frac{1}{({2}n)^r}\sum_{k=1}^n\left(-\cot\left(-\frac{\pi}{4n}+\frac{k}{n}\pi\right)-1\right)^r
  \\
  &=\frac{E_{r-1}}{2(r-1)!}. 
  \end{align*}
   The first cumulant is $\tilde{K}_1=\frac{1}{2n}\Tr\left(\left[\begin{smallmatrix}
    0 &1-i \\
    1+i& 0
    \end{smallmatrix}\right]_n \right)=0$
     and hence the 
desired $R$-transform is
$$R_{Y}(z)=\sum_{r=0}^\infty \tilde{K}_{r+1}z^r=\frac{1}{2}\sum_{r=1}^\infty \frac{E_{r}}{r!}z^r=\frac{\tan(z)+\sec(z)-1}{2}.$$
\end{proof}

\begin{Rem}
   The above results coincide with Theorem \ref{twr:CLTCommutatorsAniticommutators}. Indeed, if we use use the scaling  appropriate  for Theorem \ref{twr:CLTCommutatorsAniticommutators}, i.e., $\frac{1}{\sqrt{2}n}$,  then \begin{align*}
 R_{\sqrt{2}Y}(z)&=\frac{\tan(\sqrt{2}z)+\sec(\sqrt{2}z)-1}{\sqrt{2}}
\intertext{by the identity $\tan(z)+\sec(z)=\frac{1+\tan(z/2)}{1-\tan(z/2)}$, we have }
&=\frac{\tan(z/\sqrt{2})}{1/\sqrt{2}-\tan(z/\sqrt{2})/\sqrt{2}}.
\end{align*}
It is interesting to compare the
power series expansion of Theorem~\ref{twr:CLTCommutatorsAniticommutators} for
$a=b=\frac{1}{\sqrt{2}}$ with
$\frac{\tan(\sqrt{2}z)+\sec(\sqrt{2}z)-1}{\sqrt{2}}$, because
it shows the identity
$$
\sum_{k=0}^{n-1}T_n^{(k+1)}=2^{n-1}E_n
.
$$
This provides a new answer to a question of Carlitz and Scoville
\cite[equ.~(2.19) on p.~418]{CarlitzScoville:1972}
who assert that ``the numbers  $\sum_{k=0}^{n-1}T_n^{(k+1)}$ are not easily
evaluated''; see \cite[Prop.~6]{Cvijovic:2011:higher} for another proof.
This sequence is catalogued as $A000828$ in Sloane's database
\cite{Sloane} and the numbers are half of the \emph{Euler numbers of type $B$}, see
\cite{Ma:2014:gamma}.
\end{Rem}
\section{Spectral radius, density, L\'evy-Khinchin representation and Bercovici-Pata bijection of the tangent laws}
 \subsection{The spectral radius of the tangent law}
\begin{prop}
\label{prop:tanlawspectralradius}
    The spectral radius of the tangent law (the limit law of
    Corollary~\ref{cor:tangentlaw}) is given by
    $$
    \rho = \frac{1}{\dottie}(1+\sqrt{1-\dottie^2}) \simeq   2.2644374158937358461
    $$
    where $\dottie\approx  0.7390851332$ is the iterated cosine constant,
    i.e.,
    the unique fixed point of the equation $x=\cos x$.
\end{prop}

\begin{proof}
  Since the moments are nonnegative, Pringsheim's theorem
  (see \cite[Sec.~7.21]{Titchmarsh:1939:theory} or
     \cite[Sec.~3.6]{Markusevich:1967:teoria1})
  implies that the principal singularity of the Cauchy
  transform lies on the positive real axis and
  the spectral radius can be computed as
  $$
  \rho = \inf_{t>0} K(t)
  $$
  see \cite[Ch.~9.C]{Woess:2000:random}.
    In order to compute the minimum of the function
  $$
  K(t) = \frac{1}{t}+\tan t
  $$
  we compute 
  the roots of its derivative
  $$
  K'(t) = -\frac{1}{t^2} + \frac{1}{\cos^2t}
  .
  $$
  The unique root satisfies the equation $\cos^2t=t^2$, i.e., $t=\pm\dottie$ and thus
  $$
  \rho = \frac{1}{\dottie} + \frac{\sin\dottie}{\cos\dottie} 
    = \frac{1}{\dottie}
    \left(
      1 + \sqrt{1-\dottie^2}
    \right)
  $$
\end{proof}

\begin{Rem}
  The number $\dottie$ (Armenian letter ``ayb'') comes up from time to time in
  the literature, starting at least back in the 19th century in the 4th edition
  of Bertrands \emph{Trait\'e d'alg\`ebre} \cite{Bertrand:1865:traite},
  continuing with numerical efforts by T.H.~Miller \cite{Miller:1890:numerical}
  and the dedicated investigation by G.B.~Arakelian \cite{Arakelian:1981:fundamental}.
  This number is well known among generations of high school students who saw
  it appear on their electronic calculators when they started to
  repeatedly press the ``cos'' button during boring math classes,
  see \cite{Kaplan:2007:dottie,Salov:2012:dottie} for discussions.
\end{Rem}

\subsection{The spectral radius of the generalized free tangent laws} 
\begin{prop}
    The spectral radius of the generalized limit law from
    Theorem~\ref{twr:CLTCommutatorsAniticommutators} for $a+ib=e^{i\alpha}$,
    where $0<\alpha<\pi$ is given by
    $$
    \rho_\alpha = \frac{1}{\dottie_\alpha}(\sin\alpha+\sin \dottie_\alpha)
    $$
    where $\dottie_\alpha$ is the unique solution $x$ of the equation
    $$
    x=\sin(\alpha-x)
    .
    $$
    The dependency of the spectral radius on the parameter $\alpha$ is shown in Figure~\ref{fig:specradius}.
\end{prop}

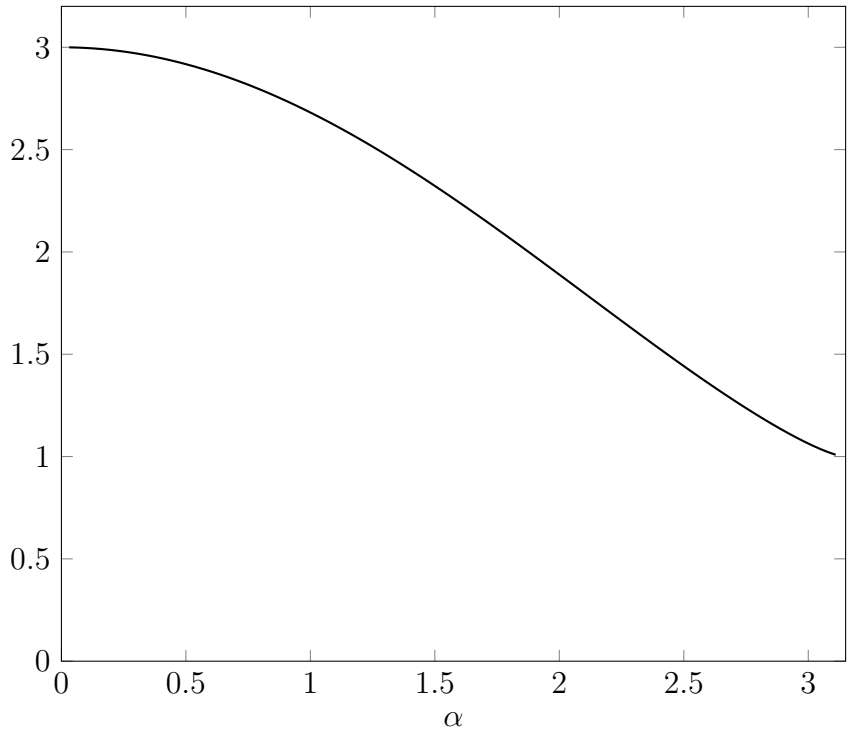
\begin{figure}
\centering{}
\tikzsetnextfilename{specradiuspos}
\begin{tikzpicture}
 \begin{axis}[xlabel={$\alpha$},
  width=0.7\textwidth,
 xmin=0,
 xmax=3.15,
 ymin=0,
 ymax=3.2
 ]
 \addplot[thick] table {specradiuspos.dat};
 \end{axis}
\end{tikzpicture}
\caption{The spectral radius of the generalized free tangent laws}
\label{fig:specradius}
\end{figure}

\begin{proof}
  We proceed as in the proof of Proposition~\ref{prop:tanlawspectralradius},
  the objective function now being
    $$
    K(t) = \frac{1}{t} + \frac{\tan bt}{b-a\tan bt}
    = \frac{1}{t} + \frac{\sin bt}{\sin\alpha\cos bt -\cos\alpha\sin bt}
    = \frac{1}{t} + \frac{\sin bt}{\sin(\alpha-bt)}
    .
    $$
    Its derivative is
    $$
    K'(t) = -\frac{1}{t^2} + \frac{\sin^2\alpha}{\sin^2(\alpha-bt)}
    $$
    and setting $x=bt$ the infimum is attained at the unique positive solution of the equation
    $$
    x=\sin(\alpha-x)
    .
    $$
\end{proof}

\subsection{The L\'evy measure of the tangent law}
The tangent function is a prominent positive definite function, see
\cite{Donoghue:1974:monotone}, and it is a fundamental example  of Nevanlinna functions.
Thus the free tangent law is $\boxplus$-infinitely divisible
and its L\'evy measure can be computed using the method from
Section~\ref{ssec:FID}.
To this end we consider the Voiculescu transform $\phi(z)=\tan(\frac{1}{z})$.
The nontangential limit of its imaginary part is
\begin{equation*}
\lim_{\epsilon \to 0^+}\Im\phi(x+i\epsilon)=\lim_{\epsilon \to
  0^+}\Im\tan\left(\frac{1}{x+i\epsilon}\right)=0 
  \quad\text{for } x\neq \frac{1}{\frac{\pi}{2}+k\pi},
\end{equation*}
and  thus  the L\'evy measure has no continuous part 
\cite[Theorem~F.6]{Schmudgen:2012:unbounded}.
In order to determine the atoms we compute the nontangential limits 
\eqref{eq:atoms}.
Now
\begin{equation*}
\lim_{\epsilon \to 0^+}i\epsilon\tan\left(\frac{1}{x+i\epsilon}\right)=0
\end{equation*}
whenever $x$ is not a pole of $\tan(1/x)$, i.e., $x\neq \frac{1}{\frac{\pi}{2}+k\pi}$, $k\in\IZ$.
On the other hand 
for $x=\frac{1}{\frac{\pi}{2}+k\pi}$ we get via de L'Hospital's rule
\begin{equation*}
\lim_
{\epsilon \to 0^+}i\epsilon\tan\left(\frac{1}{x+i\epsilon}\right)=\lim_
{\epsilon \to 0^+}\frac{i\epsilon}{\cot\left(\frac{1}{x+i\epsilon}\right)}=\lim_
{\epsilon \to 0^+}\frac{i}{\frac{-1}{\cos^2\left(\frac{1}{x+i\epsilon}\right)}\frac{-i}{(x+i\epsilon)^2}}=x^2.
\end{equation*}
Finally from \eqref{eq:atoms} we infer that the L\'evy measure is given by
$$
\nu(\{x\}) =
\begin{cases}
1& \text{ for }x=\frac{2}{n\pi} \text{ with  }n\in\IZ \text{ odd},\\
0 & \text{otherwise,}
\end{cases}
$$
which is recently confirmed by Jurek \cite{Jurek:2020}. 
Alternatively, this result can be verified as follows. 
From the well-known identity
$\sum_{n\in \IN\text{ odd}} \frac{1}{n^2}=\frac{\pi^2}{8}$ we conclude that
the tangent distribution has free characteristic triplet $(0,0,\nu)$ and we have
\begin{align*}
\Rtrans_{\mu}(z)&=\int_{\R}\left(\frac{1}{1-xz}-1-xz\mathbf{1}_{\{\abs{x}<1\}}(x)\right)\dd{\nu(x)}
\\&=\int_{\R}\left(\frac{(xz)^2}{1-xz}\right)\dd{\nu(x)}\\
&= \sum_{n\in \IZ\text{ odd}} \frac{1}{1-\frac{2z}{n\pi}} \frac{4z^2}{n^2\pi^2}\\
&= \sum_{n\in \IN\text{ odd}} \frac{2}{1-\frac{4z^2}{n^2\pi^2}} \frac{4z^2}{n^2\pi^2}\\
&= \sum_{n\in \IN\text{ odd}} \frac{8z^2}{n^2\pi^2-4z^2}.
\end{align*}
Now Euler's well known partial fraction
expansion of the contangent function \cite[Ch.~25]{AignerZiegler:2014:BOOK5ed}
$$
\cot z = \frac{1}{z} + \sum_{k=1}^\infty \frac{2z}{z^2-k^2\pi^2}
$$
immediately yields a similar expansion for the tangent function
$$
\tan z = \cot z  - 2\cot 2z = \sum_{k=1}^\infty \frac{8z}{(2k-1)^2\pi^2-4z^2}
$$
for $z\neq 0 $ and thus indeed $\frac{1}{z}\Rtrans_{\mu}(z) = \tan(z)$.

\subsection{The L\'evy measure of the generalized tangent laws} 
 The corresponding L\'evy measure in the general case is supported on the points $x=\frac{b}{\arctan(b/a)+ k\pi} \text{ for   }k\in\IZ$, with weight $1$. This follows from the fact that $\lim_
{\epsilon \to
  0^+}i\epsilon\frac{\tan(\frac{b}{x+i\epsilon})}{b-a\tan(\frac{b}{x+i\epsilon})}=x^2$.
 We leave the formal proof to the reader.

\subsection{The density of the tangent law}
 The free characteristic triplet of the free tangent law is
 $a=0$ and $\nu(\R)=\infty$ and it follows 
 from  the criterion \cite[Theorem 3.4 part (1)]{HasebeSakuma2017}
 that the free tangent law is absolutely continuous  with respect to Lebesgue
 measure.  Moreover, Huang (see
 \cite[Theorem 3.10]{Huang2012} or \cite{Huang2015}) derived a formula for the
 absolutely continuous part $\mu^{ac}$ by using the  transform
 $F_\mu^{-1}(z)=z +\tan(\frac{1}{z})$. Define a continuous map
 on $\R$ by
\begin{align*}
  v_{\mu}(x):&=\inf \{y>0\mid \Im(F_\mu^{-1}(x+iy))>0\} \\
             &=\inf \Big \{y>0\mid
 y -
\frac{\sinh\frac{2y}   {x^2+y^2}
   }{
      \cosh\frac{2y}{x^2+y^2} + \cos\frac{2x}{x^2+y^2}}
>0\Big \}.
\end{align*} 
Huang proved that we can define $\psi_\mu(x)=F_\mu^{-1}(x+iv_{\mu}(x)), \text{ for }x\in\R$,
which is a homeomorphism of $\R$ and then we have 
$$\frac{\dd{\mu^{ac}}}{\dd{x}}(\psi_\mu(x))=\frac{v_{\mu}(x)}{\pi(x^2+v_{\mu}^2(x))}.$$
The densities of the free tangent law and the free zigzag laws are shown in
Figures~\ref{fig:tanlaw} and~\ref{fig:ziglaw}; the densities
of the generalized free tangent laws for $0<\alpha<\pi$ are shown in Figure~\ref{fig:tanlaw3D}.
For comparison we also provide some histograms of the empirical spectral measures of the
random matrix model from Corollary~\ref{corrandommatrixmodel2}.

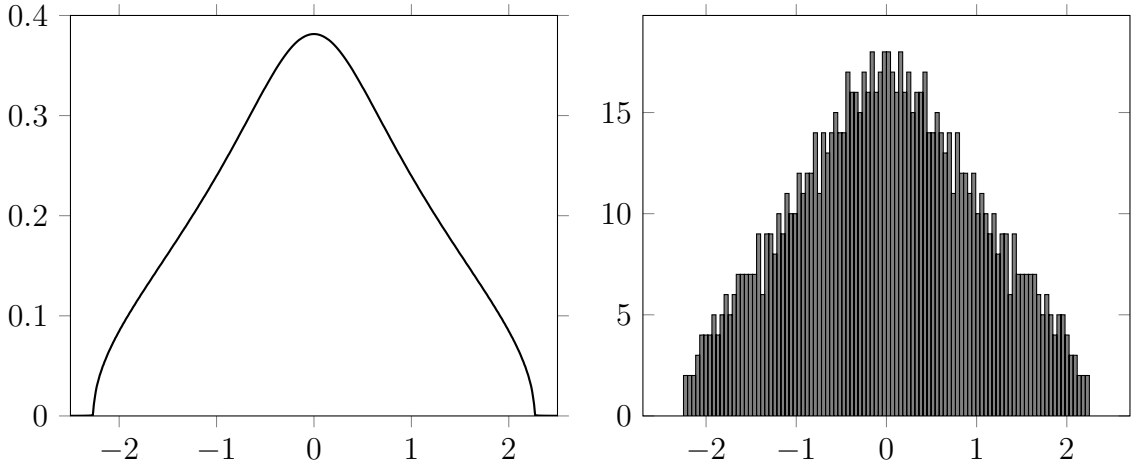
\begin{figure}
\centering{}
\tikzsetnextfilename{tanlaw_fricas_500}
\begin{tikzpicture}
 \begin{axis}[   width=0.47\textwidth,tick align=outside,
   xmin=-2.5,
   xmax=2.5,
   ymin=0,
   ymax=0.4,
   ]
   \addplot[thick] table {tanlaw_fricas_500.dat};
\end{axis}
\end{tikzpicture}
\tikzsetnextfilename{tanlaw_histogram_1000}
\begin{tikzpicture}
\begin{axis}[
   width=0.47\textwidth,
    ybar,
    ymin=0
]
\addplot +[fill=gray,draw=black,
    hist={
        bins=100,
        data min=-2.25,
        data max=2.25
    }   
] table [y index=0] {clt-tangent-1000x1000.dat};
\end{axis}
\end{tikzpicture}
  \caption{Density and histogram ($M=1000$, $N=1000$) of the free tangent law}
  \label{fig:tanlaw}
\end{figure}

\begin{figure}
\centering{}
\tikzsetnextfilename{ziglaw_fricas_500}
\begin{tikzpicture}
  \begin{axis}[
    width=0.47\textwidth,tick align=outside,
    ymin=0,
              ]
  \addplot[thick] table {ziglaw_fricas_500.dat};               
  \end{axis}
\end{tikzpicture}
\tikzsetnextfilename{zigzag_histogram_1000}
\begin{tikzpicture}
\begin{axis}[
   width=0.47\textwidth,
    ybar,
    ymin=0
]
\addplot +[fill=gray,draw=black,
    hist={
        bins=100,
        data min=-1.5,
        data max=2
    }   
] table [y index=0] {clt-zigzag2-1000x1000.dat};
\end{axis}
\end{tikzpicture}
  \caption{Density and histogram ($M=1000$, $N=1000$) of the free zigzag law}
  \label{fig:ziglaw}
\end{figure}
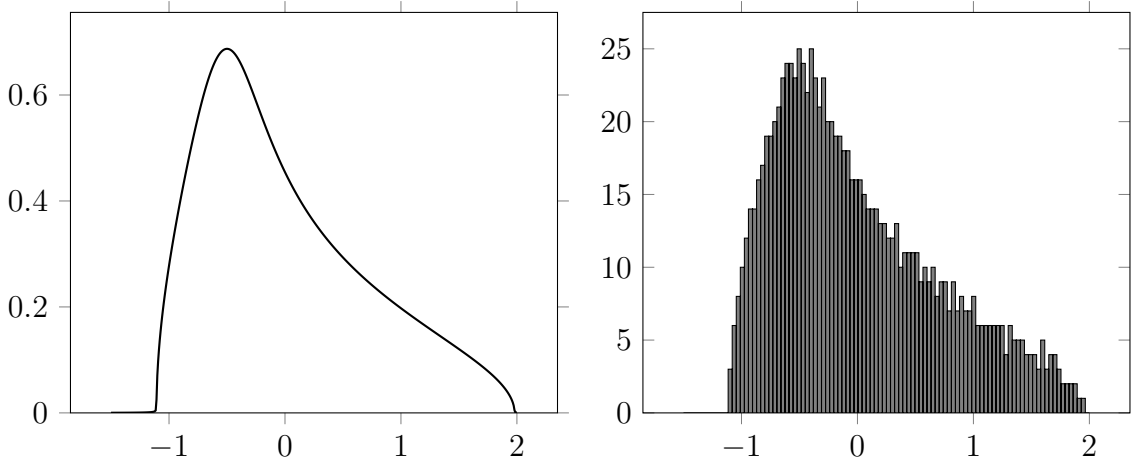

\begin{figure}
\centering{}

\includegraphics{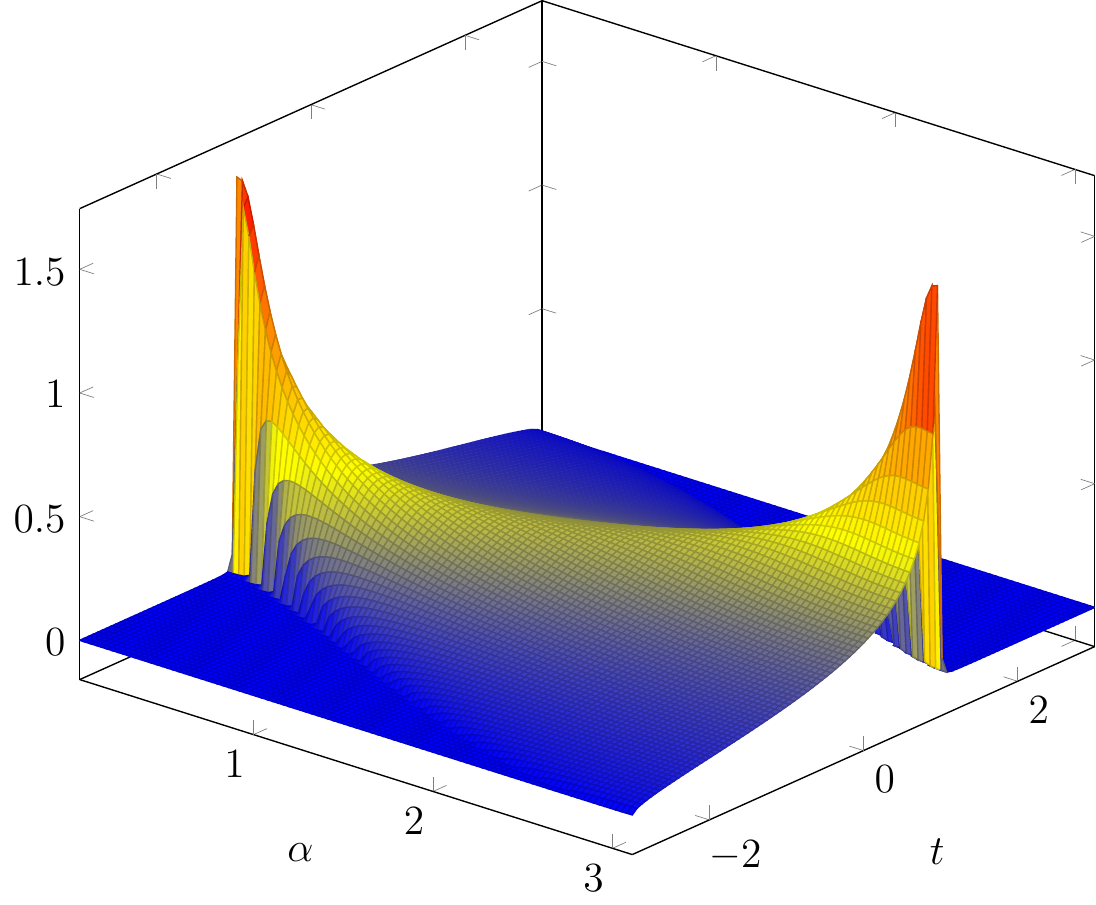}
  \caption{Densities of the generalized free tangent laws depending on $\alpha$}
  \label{fig:tanlaw3D}
\end{figure}

\subsection{Bercovici-Pata bijection}
An important connection between free and classical infinite divisibility was
established by Bercovici and Pata \cite{BercoviciPata:1999} in the form of a bijection $\Lambda$ from the class
of classical infinitely divisible laws to the class of free infinitely divisible laws. 
The easiest way to define the B-P bijection is as follows. Let $\mu$ be a probability
measure in $ID(\ast)$ having  all moments, and consider its sequence $c_n$ of classical cumulants. Then the map  $\Lambda$ can be defined as the
mapping that sends $\mu$ to the probability measure on $\R$ with free cumulants
$c_n$. 

The inverse image of the free tangent law under the Bercovici-Pata bijection
has the following characteristic function
\begin{align*}
  \log\E(\exp(zX))
  &=\sum_{n=1}^\infty (-1)^{n+1}\frac{2^{2n}(2^{2n}-1){B_{2n}}}{(2n)!}\frac{z^{2n}}{(2n)!}
   \intertext{and using Euler's identity
    $\zeta(2n)=\frac{(-1)^{n+1}(2\pi)^{2n}{B_{2n}}}{2(2n)!}$ this is }
&=2\sum_{n=1}^\infty\frac{ \zeta(2n)}{(2n)!}\left(\frac{2z}{\pi}\right)^{2n}-2\sum_{n=1}^\infty \frac{\zeta(2n)}{(2n)!}\left(\frac{z}{\pi}\right)^{2n}
  \intertext{now using the expansion
  $$\sum_{n=1}^\infty\frac{
  \zeta(2n)}{(2n)!}z^{2n}=\sum_{n=1}^\infty\frac{ \sum_{k=1}^\infty
  1/k^{2n}}{(2n)!}z^{2n}=\sum_{k=1}^\infty \sum_{n=1}^\infty\frac{
  1}{(2n)!}\left(\frac{z}{k}\right)^{2n}=\sum_{k=1}^\infty( \cosh(z/k)-1)$$
  we get further
  }
&=2\sum_{n=1}^\infty \left( \cosh\left(\frac{2z}{\pi n}\right)-1\right)-2\sum_{n=1}^\infty \left( \cosh\left(\frac{z}{\pi n}\right)-1\right)
  \\
  &=2\sum_{n\in \N \text{ odd}} \left( \cosh\left(\frac{2z}{\pi n}\right)-1\right).
  \end{align*}
Thus by using $\cosh(it)=\cos(t)$,  we obtain the characteristic function 
\begin{align*}
\E(\exp(itX))=\exp\left[2\sum_{n\in \N \text{ odd}} \left( \cos\left(\frac{2t}{\pi n}\right)-1\right)\right]=\prod_{n\in \N \text{ odd}}\exp\left[2\left( \cos\left(\frac{2t}{\pi n}\right)-1\right)\right]
\end{align*}
Note that $\exp(2 \cos(t)-2)$ is the characteristic function of Skellam distribution $X$, i.e., $$P(X=k)=P(X=-k)=\frac{I_k(2)}{\exp(2)}\text{ for } k\in\N,$$  where $I_n$ is $n$-th modified Bessel function of the first kind see  \cite{Skellam}. 
Hence $\exp\left[2\left( \cos\left(\frac{2t}{\pi n}\right)-1\right)\right]$ is
the characteristic function of the random variable $\frac{2}{n\pi}X$
and we conclude that the classical distribution corresponding to the free tangent law
under the  B-P bijection
is the law of the random variable $\sum_{n\in \N \text{ odd}}\tilde{X}_n$, where $\tilde{X}_n$ are
independent random variables such that   $\tilde{X}_n$ has the same
distribution as $\frac{2}{n\pi}X$, $ n\in \N \text{ odd}.$

\section{Concluding Remarks}
\subsection{Sums of anticommutators}
 Instead of the sums of commutators one can also
consider sums  of the anti-commutators
$\frac{1}{n}\sum_{i<j}^n(\X_i\X_j+\X_j\X_i)$ for sequence of standard free
  semicircular variables.
Contrary to the case $n=1$, where the distribution of the anticommutator $XY+YX$ coincides with
the distribution of the commutator $i(XY-YX)$ \cite[Remark~19.8 (3)]{NicaSpeicher:2006}, this leads to new distributions
for $n\geq3$ and subsequently also in the limit as $n$ tends to infinity.
Indeed the spectrum of
the corresponding matrix $A_n=\left[\begin{smallmatrix}
    0&1\\
    1& 0
  \end{smallmatrix}\right]_n$ consists of the two eigenvalues $\lambda=-1$
and $\lambda=n-1$ with respective multiplicities   $n-1$ and $1$.
Thus in the limit the $r$th cumulant
  is equal to
  \begin{equation*}
    \lim_{n\to \infty}\frac{(-1)^r(n-1)+(n-1)^r}{n^r}=
    \begin{cases}
      1 &\text{for }r\neq 1 \\
      0 & \text{for }r=1.
    \end{cases}
  \end{equation*}

  This corresponds to the Marchenko-Pastur (or free Poisson) distribution.
  Observe that we can reconstruct the Marchenko-Pastur  distribution from
  Theorem \ref{twr:CLTCommutatorsAniticommutators} by passing
  to the limit
  \begin{equation*}
    \lim_{b\to 0 }\frac{\tan(bz)}{b-\sqrt{1-b^2}\tan(bz)}=\lim_{b\to 0}\frac{\frac{\sin(bz)}{bz}}{\frac{\cos(bz)}{z}-\sqrt{1-b^2}\frac{\sin(bz)}{bz}}=\frac{z}{1- z},
  \end{equation*}
  which is the $R$-transform of the free Poisson distribution.

Such  interpolations  have  attracted  some  attention in  connection 
with random matrices. 
 As an application of Corollary \ref{corrandommatrixmodel2} we 
present  an interpolation on the unit circle $w=e^{i\alpha}$, $\alpha\in[0,2\pi)$ 
between 
the Marchenko-Pastur law 
\cite{Pastur:1967} $\alpha=0$, free tangent law $\alpha=\frac{\pi}{2}$ and free zigzag law $\alpha=\frac{\pi}{4}$ in the context of random matrices
 $$ 
    \lim_{M\to\infty}
  \lim_{N\to\infty} \frac{1}{M}X_{N\times NM}\Big[\left[\begin{smallmatrix}
    0&w\\
    \overline{w}& 0
    \end{smallmatrix}\right]_M\otimes P_N\Big]X_{N\times NM}^\ast 
    =Y$$
     $\text{where } R_{Y}(z)=\frac{\tan(z\Im w )}{\Im w-\Re w\tan(z\Im w )}$.
Thus we  are led to measures which might be called  generalized Marchenko-Pastur laws.

\subsection{The trace method for tangent numbers and the Riemann zeta function}

  Propositions~\ref{cor:tangentlaw} and \ref{twr:zigzag}  lead
to another new  fact about the tangent numbers $T_n$, the Euler zigzag numbers $E_n$, 
the Riemann zeta function and the Bernoulli numbers for  even values, namely 
\begin{align*}
 T_{2k-1}&=\lim_{n\to\infty}\frac{(2k-1)!\Tr\big( \left[\begin{smallmatrix}
    0 &i\\
    -i& 0
    \end{smallmatrix}\right]_n^{2k}\big)}{n^{2k}},\text{ }E_k=\lim_{n\to\infty}\frac{k!\Tr\big( \left[\begin{smallmatrix}
    0 &1+i\\
    1-i& 0
    \end{smallmatrix}\right]_n^{k+1}\big)}{2^kn^{k+1}}, 
    \\
    \zeta(2k)&=\lim_{n\to\infty}\frac{\pi ^{2k}\Tr\big( \left[\begin{smallmatrix}
    0 &i\\
    -i& 0
    \end{smallmatrix}\right]_n^{2k}\big)}{ 2n^{2k}(2^{2k}-1)}, B_{2k}=\lim_{n\to\infty}\frac{(2k)!\Tr\big( \left[\begin{smallmatrix}
    0 &i\\
    -i& 0
    \end{smallmatrix}\right]_n^{2k}\big)}{(-1)^{k+1}2^{2k}(2^{2k}-1)n^{2k}}\quad\text{for }k\in\N.
\end{align*}
Approximation of
the values of the Riemann zeta function
for even integers is a popular theme,
see \cite{Williams1971,Apostol1973,Cvijovic2003}.
It would be particularly interesting to obtain approximations for odd integers
as well, but for this one would have to compute the singular values of the
matrix $A_n$.

\bibliographystyle{amsplain}

\begin{thebibliography}{10}

\bibitem{AignerZiegler:2014:BOOK5ed}
Martin Aigner and G\"unter~M. Ziegler, \emph{Proofs from {T}he {B}ook}, fifth
  ed., Springer-Verlag, Berlin, 2014.

\bibitem{Andre1880}
D\'{e}sir\'{e} Andr\'{e}, \emph{D\'{e}veloppements, par rapport au module, des
  fonctions elliptiques {$\lambda (x),\;\mu (x)$} et de leurs puissances}, Ann.
  Sci. \'{E}cole Norm. Sup. (2) \textbf{9} (1880), 107--118.

\bibitem{Apostol1973}
Tom~M. Apostol, \emph{Another elementary proof of {E}uler's formula for {$\zeta
  (2n)$}}, Amer. Math. Monthly \textbf{80} (1973), 425--431.

\bibitem{Arakelian:1981:fundamental}
Grant~B. Arakelian, \emph{{\cyr Fundamental\cprime{}nye bezrazmernye velichiny}
  [{T}he fundamental dimensionless values]}, Armenian National Academy of
  Sciences, Erevan, 1981 (Russian).

\bibitem{ArizmendiHasebe:2013}
Octavio Arizmendi and Takahiro Hasebe, \emph{On a class of explicit
  {C}auchy-{S}tieltjes transforms related to monotone stable and free {P}oisson
  laws}, Bernoulli \textbf{19} (2013), no.~5B, 2750--2767.

\bibitem{Arlinskii}
Yuri Arlinskii, Sergey Belyi, and Eduard Tsekanovskii, \emph{Conservative
  realizations of {H}erglotz-{N}evanlinna functions}, Operator Theory: Advances
  and Applications, vol. 217, Birkh\"{a}user/Springer Basel AG, Basel, 2011.

\bibitem{Arnold:1991}
V.~I. Arnol'd, \emph{Bernoulli-{E}uler updown numbers associated with function
  singularities, their combinatorics and arithmetics}, Duke Math. J.
  \textbf{63} (1991), no.~2, 537--555.

\bibitem{Arnold:1992:snakes}
\bysame, \emph{Snake calculus and the combinatorics of the {B}ernoulli, {E}uler
  and {S}pringer numbers of {C}oxeter groups}, Uspekhi Mat. Nauk \textbf{47}
  (1992), no.~1(283), 3--45, 240.

\bibitem{BarndorffnielsenThorbjornsen:2006}
Ole~E. Barndorff-Nielsen and Steen Thorbj{\o}rnsen, \emph{Classical and free
  infinite divisibility and {L}\'evy processes}, Quantum independent increment
  processes. {II}, Lecture Notes in Math., vol. 1866, Springer, Berlin, 2006,
  pp.~33--159.

\bibitem{BercoviciPata:1999}
Hari Bercovici and Vittorino Pata, \emph{Stable laws and domains of attraction
  in free probability theory}, Ann. of Math. (2) \textbf{149} (1999), no.~3,
  1023--1060, With an appendix by Philippe Biane.

\bibitem{VoiculescuBercovici1992}
Hari Bercovici and Dan Voiculescu, \emph{L\'{e}vy-{H}in\v{c}in type theorems
  for multiplicative and additive free convolution}, Pacific J. Math.
  \textbf{153} (1992), no.~2, 217--248.

\bibitem{BercoviciVoiculescu:1993:unbounded}
\bysame, \emph{Free convolution of measures with unbounded support}, Indiana
  Univ. Math. J. \textbf{42} (1993), no.~3, 733--773.

\bibitem{Bertrand:1865:traite}
J.~Bertrand, \emph{Trait\'e d'alg\`ebre}, 4th ed., Hachette, Paris, 1865.

\bibitem{BianePitmanYor2001}
Philippe Biane, Jim Pitman, and Marc Yor, \emph{Probability laws related to the
  {J}acobi theta and {R}iemann zeta functions, and {B}rownian excursions},
  Bull. Amer. Math. Soc. (N.S.) \textbf{38} (2001), no.~4, 435--465.

\bibitem{Boyadzhiev:2007}
Khristo~N. Boyadzhiev, \emph{Derivative polynomials for tanh, tan, sech and sec
  in explicit form}, Fibonacci Quart. \textbf{45} (2007), no.~4, 291--303
  (2008).

\bibitem{Carlitz:1975:permutations}
L.~Carlitz, \emph{Permutations, sequences and special functions}, SIAM Rev.
  \textbf{17} (1975), 298--322.

\bibitem{CarlitzScoville:1972}
L.~Carlitz and Richard Scoville, \emph{Tangent numbers and operators}, Duke
  Math. J. \textbf{39} (1972), 413--429.

\bibitem{ColinsHasebeSakuma2018}
Benoit Collins, Takahiro Hasebe, and Noriyoshi Sakuma, \emph{Free probability
  for purely discrete eigenvalues of random matrices}, J. Math. Soc. Japan
  \textbf{70} (2018), no.~3, 1111--1150.

\bibitem{Comtet}
Louis Comtet, \emph{Advanced combinatorics}, enlarged ed., D. Reidel Publishing
  Co., Dordrecht, 1974, The art of finite and infinite expansions.

\bibitem{Cvijovic:2009:derivative}
Djurdje Cvijovi\'{c}, \emph{Derivative polynomials and closed-form higher
  derivative formulae}, Appl. Math. Comput. \textbf{215} (2009), no.~8,
  3002--3006.

\bibitem{Cvijovic:2011:higher}
\bysame, \emph{Higher-order tangent and secant numbers}, Comput. Math. Appl.
  \textbf{62} (2011), no.~4, 1879--1886.

\bibitem{Cvijovic2003}
Djurdje Cvijovi\'{c}, Jacek Klinowski, and H.~M. Srivastava, \emph{Some
  polynomials associated with {W}illiams' limit formula for {$\zeta(2n)$}},
  Math. Proc. Cambridge Philos. Soc. \textbf{135} (2003), no.~2, 199--209.

\bibitem{Donoghue:1974:monotone}
William~F. Donoghue, Jr., \emph{Monotone matrix functions and analytic
  continuation}, Springer-Verlag, New York-Heidelberg, 1974, Die Grundlehren
  der mathematischen Wissenschaften, Band 207.

\bibitem{EjsmontLehner:2017}
Wiktor Ejsmont and Franz Lehner, \emph{Sample variance in free probability}, J.
  Funct. Anal. \textbf{273} (2017), no.~7, 2488--2520.

\bibitem{EjsmontLehner:2019:commutators}
\bysame, \emph{Sums of commutators in free probability}, 2020, submitted,
  arXiv:2002.06051.

\bibitem{EjsmontLehner:2019:cotangent}
\bysame, \emph{The trace method for cotangent sums}, 2020, preprint,
  arXiv:2002.06052.

\bibitem{Flajolet1980}
P.~Flajolet, \emph{Combinatorial aspects of continued fractions}, Discrete
  Math. \textbf{32} (1980), no.~2, 125--161.

\bibitem{Gesztesy}
Fritz Gesztesy and Eduard Tsekanovskii, \emph{On matrix-valued {H}erglotz
  functions}, Math. Nachr. \textbf{218} (2000), 61--138.

\bibitem{GrahamKnuthPatashnik:1994}
Ronald~L. Graham, Donald~E. Knuth, and Oren Patashnik, \emph{Concrete
  mathematics}, second ed., Addison-Wesley Publishing Company, Reading, MA,
  1994.

\bibitem{HasebeSakuma2017}
Takahiro Hasebe and Noriyoshi Sakuma, \emph{Unimodality for free {L}\'evy
  processes}, Ann. Inst. Henri Poincar\'e Probab. Stat. \textbf{53} (2017),
  no.~2, 916--936.

\bibitem{Hoffman:1995}
Michael~E. Hoffman, \emph{Derivative polynomials for tangent and secant}, Amer.
  Math. Monthly \textbf{102} (1995), no.~1, 23--30.

\bibitem{Hoffman:1999:derivative}
\bysame, \emph{Derivative polynomials, {E}uler polynomials, and associated
  integer sequences}, Electron. J. Combin. \textbf{6} (1999), Research Paper
  21, 13.

\bibitem{Huang2012}
Hao-Wei Huang, \emph{Supports, regularity and $\boxplus$-infinite divisibility
  for measures of the form $(\mu^{\boxplus p})^{\uplus q}$}, 2012, preprint,
  arXiv:1209.5787.

\bibitem{Huang2015}
\bysame, \emph{Supports of measures in a free additive convolution semigroup},
  Int. Math. Res. Not. IMRN (2015), no.~12, 4269--4292.

\bibitem{James:1958}
G.~S. James, \emph{On moments and cumulants of systems of statistics},
  Sankhy\=a \textbf{20} (1958), 1--30.

\bibitem{Jurek:2020}
Zbigniew Jurek, \emph{Remarks on a tangent function from a probability point of
  view}, 2020, submitted, arXiv:2006.04477.

\bibitem{Kaplan:2007:dottie}
Samuel~R. Kaplan, \emph{The {D}ottie number}, Math. Mag. \textbf{80} (2007),
  no.~1, 73--74.

\bibitem{KrawczykSpeicher:2000}
Bernadette Krawczyk and Roland Speicher, \emph{Combinatorics of free
  cumulants}, J. Combin. Theory Ser. A \textbf{90} (2000), no.~2, 267--292.

\bibitem{LeonovShiryaev:1959}
V.~P. Leonov and A.~N. Shiryaev, \emph{On a method of calculation of
  semi-invariants}, Theor. Prob. Appl. \textbf{4} (1959), 319--328.

\bibitem{Ma:2014:gamma}
Shi-Mei Ma, \emph{On {$\gamma$}-vectors and the derivatives of the tangent and
  secant functions}, Bull. Aust. Math. Soc. \textbf{90} (2014), no.~2,
  177--185.

\bibitem{Markusevich:1967:teoria1}
A.~I. Marku{\v s}evi{\v c}, \emph{{\cyr {T}eoriya analiticheskikh funktsi\u\i}
  [{T}he theory of analytic functions]}, Second edition, Izdat. ``Nauka'',
  Moscow, 1967.

\bibitem{Pastur:1967}
V.~A. Mar\v{c}enko and L.~A. Pastur, \emph{Distribution of eigenvalues in
  certain sets of random matrices}, Mat. Sb. (N.S.) \textbf{72 (114)} (1967),
  507--536.

\bibitem{Miller:1890:numerical}
T.~H. Miller, \emph{On the numerical values of the roots of the equation $\cos
  x = x$}, Proc. Edinb. Math. Soc. \textbf{9} (1890), 80--83.

\bibitem{MingoSpeicher:2017}
James~A. Mingo and Roland Speicher, \emph{Free probability and random
  matrices}, Fields Insitute Monographs, vol.~35, Springer, 2017.

\bibitem{NicaSpeicher:2006}
Alexandru Nica and Roland Speicher, \emph{Lectures on the combinatorics of free
  probability}, London Mathematical Society Lecture Note Series, vol. 335,
  Cambridge University Press, Cambridge, 2006.

\bibitem{SaitohYoshida:2001}
Naoko Saitoh and Hiroaki Yoshida, \emph{The infinite divisibility and
  orthogonal polynomials with a constant recursion formula in free probability
  theory}, Probab. Math. Statist. \textbf{21} (2001), no.~1, Acta Univ.
  Wratislav. No. 2298, 159--170.

\bibitem{Salov:2012:dottie}
V.~{Salov}, \emph{{Inevitable Dottie Number. Iterals of cosine and sine}},
  November 2012, arXiv:1212.1027.

\bibitem{Schmudgen:2012:unbounded}
Konrad Schm\"{u}dgen, \emph{Unbounded self-adjoint operators on {H}ilbert
  space}, Graduate Texts in Mathematics, vol. 265, Springer, Dordrecht, 2012.

\bibitem{Skellam}
J.~G. Skellam, \emph{The frequency distribution of the difference between two
  {P}oisson variates belonging to different populations}, J. Roy. Statist. Soc.
  (N.S.) \textbf{109} (1946), 296.

\bibitem{Sloane}
N.~J.~A. Sloane, \emph{The on-line encyclopedia of integer sequences},
  published electronically at \url{https://oeis.org}, 2019.

\bibitem{Speicher:1992:central}
Roland Speicher, \emph{A noncommutative central limit theorem}, Math. Z.
  \textbf{209} (1992), no.~1, 55--66.

\bibitem{StanleyVol2}
Richard~P. Stanley, \emph{Enumerative combinatorics. {V}ol. 2}, Cambridge
  Studies in Advanced Mathematics, vol.~62, Cambridge University Press,
  Cambridge, 1999.

\bibitem{Stanley2010}
\bysame, \emph{A survey of alternating permutations}, Combinatorics and graphs,
  Contemp. Math., vol. 531, Amer. Math. Soc., Providence, RI, 2010,
  pp.~165--196.

\bibitem{Steinmetz}
Norbert Steinmetz, \emph{Nevanlinna theory, normal families, and algebraic
  differential equations}, Universitext, Springer, Cham, 2017.

\bibitem{Titchmarsh:1939:theory}
E.~C. Titchmarsh, \emph{The theory of functions}, second ed., Oxford University
  Press, Oxford, 1939.

\bibitem{VoiculescuDykemaNica:1992}
D.~V. Voiculescu, K.~J. Dykema, and A.~Nica, \emph{Free random variables}, CRM
  Monograph Series, vol.~1, American Mathematical Society, Providence, RI,
  1992.

\bibitem{Voiculescu:1985}
Dan Voiculescu, \emph{Symmetries of some reduced free product
  {$C^\ast$}-algebras}, Operator algebras and their connections with topology
  and ergodic theory ({B}u\c steni, 1983), Lecture Notes in Math., vol. 1132,
  Springer, Berlin, 1985, pp.~556--588.

\bibitem{Voiculescu:1986}
\bysame, \emph{Addition of certain noncommuting random variables}, J. Funct.
  Anal. \textbf{66} (1986), no.~3, 323--346.

\bibitem{Voiculescu:1991}
\bysame, \emph{Limit laws for random matrices and free products}, Invent. Math.
  \textbf{104} (1991), no.~1, 201--220.

\bibitem{Williams1971}
Kenneth~S. Williams, \emph{On {$\sum _{n=1}^{\infty }\ (1/n^{2k})$}}, Math.
  Mag. \textbf{44} (1971), 273--276.

\bibitem{Woess:2000:random}
Wolfgang Woess, \emph{Random walks on infinite graphs and groups}, Cambridge
  Tracts in Mathematics, vol. 138, Cambridge University Press, Cambridge, 2000.

\end{thebibliography}

\providecommand{\bysame}{\leavevmode\hbox to3em{\hrulefill}\thinspace}
\providecommand{\MR}{\relax\ifhmode\unskip\space\fi MR }
\providecommand{\MRhref}[2]{  \href{http://www.ams.org/mathscinet-getitem?mr=#1}{#2}
}
\providecommand{\href}[2]{#2}

\end{document}